\documentclass[10pt]{amsart} 
\usepackage{amsmath,amssymb,bm, color} 
\usepackage[foot]{amsaddr}

  \usepackage{amsmath}
\usepackage{graphicx,float} 

\usepackage{mathrsfs}
\usepackage{bbm}
\usepackage{bm}
 \usepackage{amsfonts,amssymb}
 \usepackage{multirow}
\usepackage{lineno}
 \usepackage{color}

\numberwithin{equation}{section}
 
 \newtheorem{theorem}{Theorem}[section]
 \newtheorem{lemma}[theorem]{Lemma}

 \def\3bar{{|\hspace{-.02in}|\hspace{-.02in}|}}

\def\T{{\mathcal{T}}}

\def\bphi{\boldsymbol{\phi}}

\def\beta{\boldsymbol{\eta}}

\def\cal#1{{\mathcal #1}}
\def\pT{{\partial T}}

\def\bw{{\mathbf{w}}}
\def\bf{{\mathbf{f}}}
\def\bg{{\mathbf{g}}}
\def\bu{{\mathbf{u}}}
\def\bv{{\mathbf{v}}}
\def\bn{{\mathbf{n}}}

\def\bvarphi{{\boldsymbol{\varphi}}}

\newtheorem{remark}{Remark}[section]
\newtheorem{algorithm}{Stabilizer-Free WG Algorithm}[section]

\numberwithin{equation}{section}

\def\3bar{{|\hspace{-.02in}|\hspace{-.02in}|}}
\def\p#1{\begin{pmatrix}#1\end{pmatrix}}
 \def\cal#1{\mathcal{#1}}
\def\ad#1{\begin{aligned}#1\end{aligned}}  \def\b#1{\mathbf{#1}} 
 \def\an#1{\begin{align}#1\end{align}}

\begin{document}

\title []
 {A Simple and Robust Weak Galerkin Method for the Brinkman Equations on Non-Convex Polytopal Meshes}

  \author {Chunmei Wang$\dag$}\thanks{$\dag$ Corresponding author. } 
  \address{Department of Mathematics, University of Florida, Gainesville, FL 32611, USA. }
  \email{chunmei.wang@ufl.edu}
  \thanks{The research of Chunmei Wang was partially supported by National Science Foundation Grant DMS-2136380.}

\author {Shangyou Zhang}
\address{Department of Mathematical Sciences,  University of Delaware, Newark, DE 19716, USA}   \email{szhang@udel.edu}

\begin{abstract}
 This paper presents a novel Stabilizer-Free weak Galerkin (WG) finite element method for solving the Brinkman equations without the need for conventional stabilization techniques. The Brinkman model, which mathematically blends features of both the Stokes and Darcy equations, describes fluid flow in multi-physics environments, particularly in heterogeneous porous media characterized by spatially varying permeability. In such settings, flow behavior may be governed predominantly by Darcy dynamics in certain regions and by Stokes dynamics in others.
A central difficulty in this context arises from the incompatibility of standard finite element spaces: elements stable for the Stokes equations typically perform poorly for Darcy flows, and vice versa. The primary challenge addressed in this study is the development of a unified numerical scheme that maintains stability and accuracy across both flow regimes.
To this end, the proposed WG method demonstrates a robust capacity to resolve both Stokes- and Darcy-dominated flows through a unified framework. The method supports general finite element partitions consisting of convex and non-convex polytopal elements, and employs bubble functions as a critical analytical component to achieve stability and convergence. Optimal-order error estimates are rigorously derived for the WG finite element solutions. Additionally, a series of numerical experiments is conducted to validate the theoretical findings, illustrating the method’s robustness, reliability, flexibility, and accuracy in solving the Brinkman equations.

\end{abstract}


  
\maketitle

\section{Introduction}  
This paper is devoted to the development of stable and efficient numerical methods for the Brinkman equations using the weak Galerkin (WG) finite element approach. The Brinkman equations serve as a unified model for fluid flow in heterogeneous porous media, where the permeability coefficient exhibits significant spatial variation. Such variability leads to regions where the flow is governed predominantly by Darcy’s law, and others where Stokes flow dominates. In its simplified form, the Brinkman model aims to determine the velocity field $\bu$ and the pressure field $p$ satisfying the following equations:
\begin{equation}\label{model}
 \begin{split}
  -\mu \Delta \bu+\nabla p+\mu \kappa^{-1} \bu=&\bf, \qquad\qquad \text{in}\quad \Omega,\\
 \nabla\cdot  \bu=&0,\qquad\qquad \text{in}\quad \Omega,\\
 \bu=&\bg,\qquad\qquad \text{on}\quad \partial\Omega.
 \end{split}
 \end{equation}Here, $\mu$ denotes the dynamic viscosity of the fluid, and $\kappa$ represents the permeability tensor of the porous medium, which occupies a polygonal or polyhedral domain $\Omega \subset \mathbb{R}^d$ with spatial dimension $d = 2$ or $3$. The vector field $\mathbf{f}$ corresponds to a prescribed momentum source. For simplicity and without loss of generality, we consider the model in the case where $\mathbf{g} = 0$ and $\mu = 1$.
We assume that the permeability tensor $\kappa$ is symmetric and uniformly positive definite. Specifically, there exist two positive constants $\lambda_1, \lambda_2 > 0$ such that
 $$
 \lambda_1\xi^T\xi\leq \xi^T \kappa^{-1}\xi\leq \lambda_2\xi^T\xi, \qquad \xi\in \mathbb R^d,
 $$where $\xi^T$ denotes the transpose of the vector $\xi$. For simplicity of analysis, we assume throughout the paper that $\kappa$ is constant. However, the analysis can be readily extended to accommodate variable functions without difficulty.

The variational formulation of the Brinkman problem \eqref{model} is stated as follows: Find $\bu \in [H_0^1(\Omega)]^d$ and $p \in L_0^2(\Omega)$ such that
\begin{equation}\label{weak}
 \begin{split}
     (\nabla \bu, \nabla \bv)-(\nabla \cdot\bv, p)+(\kappa^{-1} \bu, \bv)=&(\bf, \bv), \qquad \forall \bv\in [H_0^1(\Omega)]^d \\
     (\nabla \cdot\bu, q)=&0, \qquad \qquad\forall q\in L_0^2(\Omega),
 \end{split}
 \end{equation} 
  where  the Sobolev space $H_0^1(\Omega)$ is defined by $$H_0^1(\Omega)=\{w\in H^1(\Omega): w|_{\partial\Omega}=0\},$$ and the space of square-integrable functions with zero mean is given by $$L_0^2(\Omega)=\{q\in L^2(\Omega); \int_{\Omega} qdx=0\}.$$

The Brinkman equations \eqref{model} are widely employed to model fluid flow in porous media with embedded fractures. This model can also be viewed as an extension of the Stokes equations, which themselves serve as reliable approximations of the Navier–Stokes equations in the regime of low Reynolds numbers. Accurate modeling of fluid transport in such complex multiphysics environments is essential for a range of industrial and environmental applications, including the design of industrial filters, flow through open-cell foams, and fluid movement in naturally fractured or vuggy reservoirs.

In these scenarios, the permeability field often exhibits high contrast, resulting in significant spatial variation in flow velocity throughout the porous domain. From a mathematical perspective, the Brinkman model can be interpreted as a coupling of the Stokes and Darcy equations, with the governing behavior transitioning between these two regimes in different regions of the computational domain. This change in equation type presents a fundamental challenge for numerical simulation: the numerical method must remain stable and accurate across both the Darcy- and Stokes-dominated zones.

As shown in \cite{12}, standard finite element methods that are stable for one flow regime may perform poorly when applied to the other. For instance, when the flow becomes Darcy-dominated, the convergence rates of typically stable Stokes elements, such as the conforming $P_2–P_0$  element, the nonconforming Crouzeix–Raviart element, and the Mini element, tend to deteriorate. Conversely, in Stokes-dominated regimes, finite element spaces designed for Darcy flow, such as the lowest-order Raviart–Thomas element, also exhibit a loss in convergence accuracy \cite{12}.

A central challenge in the numerical solution of the Brinkman equations lies in the development of discretization schemes that are simultaneously stable for both the Darcy and Stokes regimes. This issue arises due to the fundamental difference in the nature of these two flow models and the change of type across the computational domain. In the literature, considerable effort has been devoted to addressing this challenge by modifying classical Stokes or Darcy finite element spaces to construct new elements that exhibit uniform stability for the Brinkman system. For instance, approaches based on Stokes-stable elements have been explored in \cite{b1}, while methods extending Darcy-stable elements are presented in \cite{11,12}.

The weak Galerkin (WG) finite element method offers a novel and flexible framework for the numerical approximation of partial differential equations (PDEs). It is formulated by interpreting differential operators in a weak sense, inspired by distribution theory, and is particularly well-suited for approximations involving discontinuous, piecewise polynomial functions. In contrast to classical methods, WG techniques reduce regularity requirements on trial and test spaces through the use of appropriately constructed weak derivatives and stabilizing terms.

Over the past decade, WG methods have been extensively developed and applied to a wide range of model problems, demonstrating robust performance and broad applicability in scientific computing; see, e.g., \cite{wg1, wg2, wg3, wg4, wg5, wg6, wg7, wg8, wg9, wg10, wg11, wg12, wg13, wg14, fedi, wg15, wg16, wg17, wg18, wg19, wg20, wg21, itera, wy3655}. A key feature of WG methods is their reliance on weak continuity and weak derivatives, enabling the design of schemes that naturally conform to the variational structure of PDEs. This intrinsic flexibility allows WG methods to maintain stability and accuracy across a broad class of problems, including those with complex geometries and mixed physical regimes.
In particular, WG methods have been proposed for the Brinkman equations, demonstrating promising stability and approximation properties under varying flow regimes \cite{wgbrink}.
 
This paper is the first in the literature to  introduce  a simplified formulation of the  WG  finite element method that accommodates both convex and non-convex polygonal or polyhedral elements in the finite element partition. This formulation builds upon a recently developed Stabilizer-Free WG framework, which has been successfully applied to a variety of partial differential equations, including the Poisson equation \cite{autosecon}, the biharmonic equation \cite{autobihar, autobihar2}, linear elasticity problems \cite{autoelas}, Stokes equations \cite{autostokes}, Maxwell equations \cite{automaxwell} and other PDE models \cite{autoft, autoelasinter, autoquadcurl}.
A central innovation of the proposed method lies in the elimination of explicit stabilizing terms through the use of higher-degree polynomials in the construction of weak gradient and weak divergence operators. This strategy retains the size and sparsity structure of the global stiffness matrix, while significantly reducing implementation complexity compared to traditional WG methods that rely on carefully designed stabilizers.
An important analytical tool in this framework is the use of bubble functions, which facilitate the extension of WG techniques to non-convex polytopal meshes, a notable advancement beyond existing stabilizer-free WG methods, which are generally restricted to convex element geometries.
Rigorous theoretical analysis is conducted to establish optimal-order error estimates for the proposed WG method in both the discrete $H^1$-norm and the $L^2$-norm, thereby confirming the accuracy and robustness of the approach.

 The remainder of this paper is organized as follows. Section 2 provides a concise review of the weak gradient and weak divergence operators, along with their discrete analogues. In Section 3, we introduce an Stabilizer-Free  WG scheme for the Brinkman equations that removes the need for explicit stabilization terms and supports general polytopal meshes, including non-convex elements. Section 4 is devoted to proving the existence and uniqueness of the solution for the proposed scheme. Section 5 derives the error equation associated with the WG formulation. Section 6 establishes optimal error estimates for the numerical solution in the discrete $H^1$-norm, and Section 7 extends the analysis to obtain convergence rates in the $L^2$-norm. Finally, Section 8 presents numerical experiments that demonstrate the accuracy, stability, and flexibility of the proposed method and validate the theoretical findings.

Throughout this paper, standard notations are used. Let  
 $D\subset \mathbb{R}^d$  denote an open, bounded domain with a Lipschitz continuous boundary. For any integer $s\geq 0$,  the inner product, seminorm, and norm in the Sobolev space  $H^s(D)$ are denoted by $(\cdot,\cdot)_{s,D}$, $|\cdot|_{s,D}$ and $\|\cdot\|_{s,D}$ respectively. When  $D=\Omega$, the subscript $D$ is omitted for brevity. Furthermore, when  $s=0$, the notations simplify to $(\cdot,\cdot)_D$, $|\cdot|_D$ and $\|\cdot\|_D$, respectively.

\section{Discrete Weak Gradient and Discrete Weak Divergence} 
This section reviews the definitions of the weak gradient and weak divergence operators, along with their corresponding discrete formulations, as originally introduced in \cite{autostokes}.

Let $T$ be a polytopal element with boundary  $\partial T$. A weak function on $T$  is defined as a pair  $\bv=\{\bv_0, \bv_b\}$, where $\bv_0\in [L^2(T)]^d$ represents the interior values and $\bv_b\in [L^{2}(\partial T)]^d$ represents the boundary values. Importantly, 
  $\bv_b$ is not required to coincide with the trace of $\bv_0$ on $\partial T$.  
 
The space of all weak functions on   $T$, denote by  $W(T)$, is given by
 \begin{equation*}\label{2.1}
 W(T)=\{\bv=\{\bv_0, \bv_b\}: \bv_0\in [L^2(T)]^d, \bv_b\in [L^{2}(\partial
 T)]^d\}.
\end{equation*}
 
 The weak gradient $\nabla_w\bv$ is a linear operator that maps
 $W(T)$ into the dual space of $[H^1(T)]^{d\times d}$. For any
 $\bv\in W(T)$, the weak gradient is defined by  
 \begin{equation*} 
  (\nabla_w\bv, \bvarphi)_T=-(\bv_0, \nabla \cdot \bvarphi)_T+
  \langle \bv_b,   \bvarphi \cdot \bn \rangle_{\partial T},\quad \forall \bvarphi\in [H^1(T)]^{d\times d},
  \end{equation*}
 where $\bn$  denotes the outward unit normal vector to $\partial T$, with components $n_i (i=1,\cdots, d)$.

Similarly, the weak divergence $\nabla_w\cdot \bv$ is a linear
 operator mapping 
 $W(T)$ into the dual space of  $H^1(T)$, defined as
 \begin{equation*} 
  (\nabla_w\cdot \bv, w)_T=-(\bv_0, \nabla  w)_T+
  \langle \bv_b\cdot\bn,  w \rangle_{\partial T},\quad \forall w\in H^1(T).
  \end{equation*} 
 
 For any non-negative integer $r\ge 0$, let $P_r(T)$ denote the space of polynomials of total degree at most 
 $r$.   The discrete weak gradient 
 $\nabla_{w, r, T}\bv$ is defined as the unique polynomial in $[P_r(T)]^{d\times d}$ satisfying
 \begin{equation}\label{2.4}
(\nabla_{w, r, T}\bv, \bvarphi)_T=-(\bv_0, \nabla \cdot \bvarphi)_T+
  \langle \bv_b,   \bvarphi \cdot \bn \rangle_{\partial T},\quad \forall \bvarphi \in [P_r(T)]^{d\times d}.
  \end{equation}
If  $\bv_0\in
 [H^1(T)]^d$ is smooth,  then integration by parts applied to the first term in \eqref{2.4} yields an equivalent formulation:
 \begin{equation}\label{2.4new}
(\nabla_{w, r, T}\bv, \bvarphi)_T= (\nabla \bv_0,  \bvarphi)_T+
  \langle \bv_b-\bv_0,   \bvarphi \cdot \bn \rangle_{\partial T}, \quad \forall \bvarphi \in [P_r(T)]^{d\times d}.
  \end{equation}

 The discrete weak divergence
 $\nabla_{w, r, T}\cdot\bv$   is defined as the unique polynomial in $P_r(T)$ satisfying
 \begin{equation}\label{div}
(\nabla_{w, r, T}\cdot \bv, w)_T=-(\bv_0, \nabla  w)_T+
  \langle \bv_b\cdot\bn,  w \rangle_{\partial T},\quad \forall w\in P_r(T).
  \end{equation}
Again, if $\bv_0\in
 [H^1(T)]^d$ is smooth, an integration by parts yields the equivalent expression:
 \begin{equation}\label{divnew}
(\nabla_{w, r, T}\cdot \bv, w)_T= (\nabla   \cdot\bv_0,  w)_T+
  \langle (\bv_b-\bv_0)\cdot\bn,  w \rangle_{\partial T},\quad \forall w\in P_r(T).
  \end{equation}  

\section{Stabilizer-Free Weak Galerkin Algorithms }\label{Section:WGFEM}
 Let ${\cal T}_h$ be a finite element partition of the domain
 $\Omega\subset \mathbb R^d$ into polytopal elements that satisfy the shape regularity condition described in \cite{wy3655}. Denote by ${\mathcal E}_h$ the set of all edges (in 2D) or faces (in 3D) in
 ${\cal T}_h$, and let  ${\mathcal E}_h^0={\mathcal E}_h \setminus
 \partial\Omega$ represent the collection of interior edges or faces. For each element $T\in {\cal T}_h$, let $h_T$ denote its diameter, and define the mesh size as
 $$h=\max_{T\in {\cal
 T}_h}h_T.$$

 Let   $k\geq 1$ be a fixed integer. For each   $T\in\T_h$, the local weak finite element space is defined by
 \begin{equation*}
 V(k,   T)=\{\{\bv_0,\bv_b\}: \bv_0\in [P_k(T)]^d, \bv_b\in [P_{k}(e)]^d, e\subset \partial T\}.    
 \end{equation*}
The global weak finite element space $V_h$ is then constructed by assembling the local spaces $V(k, T)$ over all  $T\in {\cal T}_h$, with the condition that the boundary component  $\bv_b$ is single-valued across interior edges or faces, i.e.,
\an{\label{V-h}
 V_h=\big\{\{\bv_0,\bv_b\}:\ \{\bv_0,\bv_b\}|_T\in V(k, T),
 \forall T\in {\cal T}_h \big\}. }
The subspace of $V_h$ consisting of functions with vanishing boundary values on $\partial\Omega$ is defined as:
$$
V_h^0=\{\bv\in V_h: \bv_b|_{\partial\Omega}=0\}.
$$

For the pressure variable, the corresponding  finite element space is given by 
\an{\label{W-h}
W_h=\{q\in L_0^2(\Omega): q|_T \in P_{k-1}(T)\}. }

For notational simplicity, the discrete weak gradient  $\nabla_{w} \bv$  and discrete weak divergence $\nabla_{w} \cdot\bv$ refer to the element-wise defined operators 
$\nabla_{w, r, T}\bv$ and $\nabla_{w, r, T} \cdot\bv$, as introduced in equations \eqref{2.4} and \eqref{div}, respectively:
\an{\label{wg-k}
(\nabla_{w} \bv)|_T & = \nabla_{w, r, T}(\bv |_T), \qquad \forall T\in \T_h,  \\
\label{wd-k}
(\nabla_{w}\cdot \bv)|_T &= \nabla_{w, r, T}\cdot (\bv |_T), \qquad \forall T\in \T_h. }
 
On each element $T\in\T_h$, let $Q_0$  denote the $L^2$ projection onto $P_k(T)$. Similarly, on each edge or face $e\subset\partial T$, let $Q_b$ denote the $L^2$ projection onto $P_{k}(e)$. Then, for any $\bv\in [H^1(\Omega)]^d$,  the $L^2$ projection  into the weak finite element space $V_h$ is defined locally by
 $$
  (Q_h \bv)|_T:=\{Q_0(\bv|_T),Q_b(\bv|_{\pT})\},\qquad \forall T\in\T_h.
$$

We now present a simplified  WG numerical scheme for solving the Brinkman equations \eqref{model}, which eliminates the need for stabilization terms.
\begin{algorithm}\label{PDWG1}
Find $\bu_h=\{\bu_0, \bu_b\} \in V_h^0$ and $p_h\in W_h$  such that  
\begin{equation}\label{WG}
\begin{split}
  (\nabla_w \bu_h, \nabla_w \bv_h)-(\nabla_w \cdot\bv_h, p_h)+(\kappa^{-1} \bu_0,\bv_0)=&(\bf, \bv_0), \qquad \forall \bv_h\in V_h^0,\\
     (\nabla_w \cdot\bu_h, q_h)=&0, \qquad\qquad \forall q_h\in W_h, 
\end{split}
\end{equation}
where   the inner product is understood as the sum over all elements:
$$
(\cdot, \cdot)=\sum_{T\in {\cal T}_h}(\cdot, \cdot)_T.
$$
\end{algorithm}

\section{Solution Existence and Uniqueness}  
Let ${\cal T}_h$ be a shape-regular finite element mesh of the domain $\Omega$. For any element $T\in {\cal T}_h$ and any function $\phi\in H^1(T)$, the following trace inequality holds  (see \cite{wy3655}):
\begin{equation}\label{tracein}
 \|\phi\|^2_{\partial T} \leq C(h_T^{-1}\|\phi\|_T^2+h_T \|\nabla \phi\|_T^2).
\end{equation}
Moreover, if $\phi$ is a polynomial function defined on $T$, a simplified version of the trace inequality applies (see \cite{wy3655}): 
\begin{equation}\label{trace}
\|\phi\|^2_{\partial T} \leq Ch_T^{-1}\|\phi\|_T^2.
\end{equation}

For any function $\bv=\{\bv_0, \bv_b\}\in V_h$, we define the following norm:
\begin{equation}\label{3norm}
\3bar \bv\3bar= \Big( \sum_{T\in {\cal T}_h}(\nabla_{w} \bv, \nabla_{ w} \bv)_T+(\kappa^{-1} \bv_0,\bv_0)_T \Big) ^{\frac{1}{2}},
\end{equation}
along with a discrete $H^1$ norm given by:
\begin{equation}\label{disnorm}
\|\bv\|_{1, h}=\Big( \sum_{T\in {\cal T}_h} \|\nabla \bv_0\|_T^2+(\kappa^{-1} \bv_0,\bv_0)_T+h_T^{-1}\|\bv_0-\bv_b\|_{\partial T}^2 \Big)^{\frac{1}{2}}.
\end{equation}

\begin{lemma}\cite{autosecon}\label{norm1}
 For $\bv=\{\bv_0, \bv_b\}\in V_h$, there exists a constant $C$ such that
 $$
 \|\nabla 
 \bv_0\|_T\leq C\|\nabla_w \bv\|_T.
 $$
\end{lemma}

\begin{remark}
Consider an element  $T\in {\cal T}_h$, which is a general (not necessarily convex) polytopal cell with $N$ edges or faces labeled  $e_1, \cdots, e_N$.   For each edge  or face    $e_i$, we define a linear function  $l_i(x)$ satisfying  $l_i(x)=0$ on $e_i$.
We define the bubble function associated with the element $T$ as:
   $$
  \Phi_B =l^2_1(x)l^2_2(x)\cdots l^2_N(x) \in P_{2N}(T).
  $$ 
  By construction, $\Phi_B$ vanishes on the boundary $\partial T$. This function can be scaled so that $\Phi_B(M) = 1$, where $M$ represents the barycenter of $T$. Moreover, there exists a subregion $\hat{T} \subset T$ such that $\Phi_B \geq \rho_0$ for some positive constant $\rho_0$. Under these conditions, we choose $\nabla _w\bv\in [P_r(T)]^d$,  where  $r = 2N + k - 1$ as stated in  Lemma \ref{norm1}.
   
   In the special case where the polytopal element $T$ is convex, the bubble function used in Lemma \ref{norm1} can be simplified to:
 $$
 \Phi_B =l_1(x)l_2(x)\cdots l_N(x).
 $$ 
This simplified bubble function also satisfies $\Phi_B = 0$ on $\partial T$, and there exists a subdomain $\hat{T} \subset T$ such that $\Phi_B \geq \rho_0$ for some constant $\rho_0 > 0$. In the convex case, we  choose $\nabla _w\bv\in [P_r(T)]^d$,  where  $r = N + k - 1$ in   Lemma \ref{norm1}.
\end{remark}

Recall that $T$ denotes a $d$-dimensional polytopal element and $e_i$ represents one of its $(d-1)$-dimensional edges or faces. For each such face $e_i$, we define the corresponding edge/face bubble function by: $$\varphi_{e_i}= \Pi_{k=1, \cdots, N, k\neq i}l_k^2(x).$$ This function satisfies two important properties:
(1)
$\varphi_{e_i}$ vanishes on every edge or face $e_k$ with $k \ne i$; (2) there exists a subregion $\widehat{e_i} \subset e_i$ where $\varphi_{e_i} \geq \rho_1$ for some constant $\rho_1 > 0$.

\begin{lemma}\cite{autostokes}\label{phi}
    Let  $\bv=\{\bv_0, \bv_b\}\in V_h$. Define $\bvarphi=(\bv_b-\bv_0)^T \bn\varphi_{e_i}$, where $\bn$ is the outward unit normal vector to  $e_i$. Then the following inequality holds:
\begin{equation}
  \|\bvarphi\|_T ^2 \leq Ch_T \int_{e_i}|\bv_b-\bv_0|^2ds.
\end{equation}
\end{lemma}

\begin{lemma}\label{normeqva} There exist constants $C_1, C_2 > 0$ such that for all $\bv=\{\bv_0, \bv_b\} \in V_h$,   the norms $\|\cdot\|_{1, h}$ and $\3bar\cdot\3bar$ are equivalent:
 \begin{equation}\label{normeq}
 C_1\|\bv\|_{1, h}\leq \3bar \bv\3bar  \leq C_2\|\bv\|_{1, h}.
\end{equation}
\end{lemma}

\begin{proof}   
Let $T$ be a (possibly non-convex) polytopal element. As defined previously, the bubble function associated with edge/face $e_i$ is
$$\varphi_{e_i}= \Pi_{k=1, \cdots, N, k\neq i}l_k^2(x).$$

To proceed, we extend the function $\bv_b$, initially defined only on the $(d-1)$-dimensional face $e_i$, to the full $d$-dimensional element $T$. This extension is given by:
 $$
   \bv_b (X)= \bv_b(\text{Proj}_{e_i} (X)),
  $$
  where $\text{Proj}{e_i}(X)$ denotes the orthogonal projection of a point $X \in T$ onto the hyperplane $H \subset \mathbb{R}^d$ containing $e_i$. When $\text{Proj}_{e_i}(X) \notin e_i$, $\bv_b$ is taken as a suitable extension from $e_i$ to $H$.

Similarly, let $\bv_{trace}$ denote the trace of $\bv_0$ on $e_i$, and extend it to the entire element $T$ in a comparable manner. For simplicity, both extensions are still denoted as $\bv_b$ and $\bv_0$, respectively.

Now, using the test function $\bvarphi=(\bv_b-\bv_0)^T \bn\varphi_{e_i}$ in \eqref{2.4new}, we obtain 
\begin{equation}\label{t3} 
 \begin{split}
  (\nabla_{w} \bv, \bvarphi)_T=&(\nabla \bv_0, \bvarphi)_T+
  {  \langle \bv_b-\bv_0,  \bvarphi\cdot \bn \rangle_{\partial T}} \\=&(\nabla \bv_0, \bvarphi)_T+ \int_{e_i} |\bv_b-\bv_0|^2 \varphi_{e_i}ds,   
 \end{split}
  \end{equation} 
 We used the following properties of the bubble function:
(1) $\varphi_{e_i} = 0$ on each $e_k$ for $k \ne i$, and
(2) there exists a subdomain $\widehat{e_i} \subset e_i$ such that $\varphi_{e_i} \geq \rho_1$ for some constant $\rho_1 > 0$.

 Applying the Cauchy–Schwarz inequality, \eqref{t3}, the domain inverse inequality from \cite{wy3655}, and Lemma \ref{phi}, we obtain:
\begin{equation*}
\begin{split}
 \int_{e_i}|\bv_b-\bv_0|^2  ds\leq &C  \int_{e_i}|\bv_b-\bv_0|^2  \varphi_{e_i}ds \\
 \leq & C(\|\nabla_w \bv\|_T+\|\nabla \bv_0\|_T)\| \bvarphi\|_T\\
 \leq & {Ch_T^{\frac{1}{2}} (\|\nabla_w \bv\|_T+\|\nabla \bv_0\|_T) (\int_{e_i}(|\bv_0-\bv_b|^2ds)^{\frac{1}{2}}}.
 \end{split}
\end{equation*}
Using Lemma \ref{norm1}, we then derive:
$$
 h_T^{-1}\int_{e_i}|\bv_b-\bv_0|^2  ds \leq C  (\|\nabla_w \bv\|^2_T+\|\nabla \bv_0\|^2_T)\leq C\|\nabla_w \bv\|^2_T.
$$
Combining this estimate with Lemma \ref{norm1}, as well as equations \eqref{3norm} and \eqref{disnorm}, we conclude:
$$
 C_1\|\bv\|_{1, h}\leq \3bar \bv\3bar.
$$
 
Next, we apply identity \eqref{2.4new}, the Cauchy–Schwarz inequality, and the trace inequality \eqref{trace}. This yields:
$$
 \Big|(\nabla_{w} \bv, \bvarphi)_T\Big| \leq \|\nabla \bv_0\|_T \|  \bvarphi\|_T+
Ch_T^{-\frac{1}{2}}\|\bv_b-\bv_0\|_{\partial T} \| \bvarphi\|_{T},
$$
which implies
$$
\| \nabla_{w} \bv\|_T^2\leq C( \|\nabla \bv_0\|^2_T  +
 h_T^{-1}\|\bv_b-\bv_0\|^2_{\partial T}).
$$
Hence, $$ \3bar \bv\3bar  \leq C_2\|\bv\|_{1, h}.$$

 This completes the proof.
 \end{proof}

  \begin{remark}
  If the polytopal element $T$ is convex, the edge/face bubble function in Lemma \ref{normeqva} can be simplified to
$$\varphi_{e_i}= \Pi_{k=1, \cdots, N, k\neq i}l_k(x).$$
It can be readily verified that: (1) $\varphi_{e_i} = 0$ on $e_k$ for $k \ne i$, and
(2) there exists a subdomain $\widehat{e_i} \subset e_i$ such that $\varphi_{e_i} \geq \rho_1$ for some constant $\rho_1 > 0$.

Therefore, Lemma \ref{normeqva} holds with the same proof under this simplified construction.
\end{remark}

Let ${\cal Q}_h$ denote the $L^2$ projection operator onto the local finite element space of piecewise polynomials of degree at most $2N + k - 1$ on non-convex elements and $N + k - 1$ on convex elements in the finite element partition.

\begin{lemma}\label{Lemma5.1} \cite{autostokes}  For any $\bu \in [H^1(T)]^d$, the following identities hold:
\begin{equation}\label{pro1}
\nabla_{w}\bu ={\cal Q}_h(\nabla \bu), 
\end{equation}
\begin{equation}\label{pro2}
\nabla_{w} \cdot \bu ={\cal Q}_h(\nabla \cdot \bu),  
\end{equation}
\begin{equation}\label{pro3}
\nabla_w\cdot Q_h\bu={\cal Q}_h(\nabla \cdot \bu).
\end{equation}
\begin{equation}\label{pro4}
\nabla_w  Q_h\bu={\cal Q}_h(\nabla  \bu).
\end{equation}
\end{lemma}

For the bilinear form $b(\cdot, \cdot)$, we establish the following inf-sup condition.
 \begin{lemma}  
     There exists a constant $C>0$,   independent of the mesh size $h$, such that for all $\zeta\in W_h$, 
     \begin{equation}\label{infsup}
         \sup_{\bv\in V_h^0} \frac{(\nabla_w \cdot \bv, \zeta)}{\3bar \bv\3bar}\geq C\|\zeta\|.
     \end{equation}
 \end{lemma}
\begin{proof}
For any given $\zeta\in W_h\subset L_0^2(\Omega)$, it is well-known (see, e.g., \cite{2, 4, 5, 6, 7}) that there exists a vector function  $\bar{\bv}\in [H_0^1(\Omega)]^d$ such that
\begin{equation} \label{eqq}
\frac{(\nabla\cdot\bar{\bv}, \zeta)}{\|\bar{\bv}\|_1}\geq C\|\zeta\|,
\end{equation}
where the constant $C>0$ depends only on the domain $\Omega$. Define $\bv=Q_h\bar{\bv}\in V_h$.  We claim that 
\begin{equation}\label{err1}
    \3bar \bv\3bar \leq C\|\bar{\bv}\|_1,
\end{equation}
for some constant $C$. To prove this, we use identity \eqref{pro4}, which gives
$$
\sum_{T\in {\cal T}_h}\|\nabla_w \bv\|_T^2=\sum_{T\in {\cal T}_h}\|\nabla_w Q_h\bar{\bv}\|_T^2=\sum_{T\in {\cal T}_h}\|\cal{Q}_h\nabla  \bar{\bv}\|_T^2\leq \sum_{T\in {\cal T}_h}\| \nabla  \bar{\bv}\|_T^2.
$$
Also, we estimate 
\begin{equation*}
    \begin{split}
    \sum_{T\in {\cal T}_h} (\kappa^ {-1}\bv_0, \bv_0)_T =\sum_{T\in {\cal T}_h}(\kappa ^ {-1}Q_0\bar{\bv}, Q_0\bar{\bv})_T \leq     \sum_{T\in {\cal T}_h} \|\bar{\bv}\|_T^2.  
    \end{split}
\end{equation*}
Combining these inequalities gives the desired bound in \eqref{err1}.

Using identity \eqref{pro3}, we have for each $T$,
$$
(\nabla_w \cdot  \bv, \zeta)_T=(\nabla_w \cdot Q_h\bar{\bv}, \zeta)_T=( \cal{Q}_h \nabla \cdot\bar{\bv}, \zeta)_T=(  \nabla\cdot\bar{\bv}, \zeta)_T.
$$
Combining the estimate above with \eqref{eqq} and \eqref{err1}, we obtain
$$
\frac{(\nabla_w \cdot \bv, \zeta)_T}{\3bar \bv\3bar} \geq C \frac{(\nabla  \cdot \bar{\bv}, \zeta)_T}{ \| \bar{\bv}\|_1 } \geq C\|\zeta\|.
$$

This completes the proof of the Lemma.

\end{proof}
 
\begin{theorem} The Stabilizer-Free WG Algorithm \ref{PDWG1}  for the Brinkman  equations \eqref{model} admits a unique solution.
\end{theorem}
\begin{proof}
Suppose there exist two distinct solutions $(\bu_h^{(1)}, p_h^{(1)})\in V_h^0\times W_h$ and $(\bu_h^{(2)}, p_h^{(2)})\in V_h^0\times W_h$ of the  Stabilizer-free WG scheme  \ref{PDWG1}. Define their difference as  $${\Xi}_{\bu_h}= \{{\Xi}_{\bu_0}, {\Xi}_{\bu_b}\}=\bu_h^{(1)}-\bu_h^{(2)}\in V_h^0, \qquad {\Xi}_{p_h}=p_h^{(1)}-p_h^{(2)}\in W_h.$$  Then the pair ${\Xi}_{\bu_h}$  and ${\Xi}_{p_h}$ satisfies the following:
\begin{equation}\label{wgstokes}
\begin{split}
  (\nabla_w {\Xi}_{\bu_h}, \nabla_w \bv_h)-(\nabla_w \cdot\bv_h, {\Xi}_{p_h})+(\kappa^{-1} {\Xi}_{\bu_0}, \bv_0)=&0, \qquad \qquad \forall \bv_h\in V_h^0,\\
     (\nabla_w \cdot {\Xi}_{\bu_h}, q_h)=&0, \qquad\qquad \forall q_h\in W_h.
\end{split}
\end{equation}
Choosing $\bv_h={\Xi}_{\bu_h}$ and $q_h={\Xi}_{p_h}$ in \eqref{wgstokes} yields 
$$\3bar {\Xi}_{\bu_h} \3bar=0.$$
Using the norm equivalence \eqref{normeq}, we conclude that $$\|{\Xi}_{\bu_h}\|_{1,h}=0,$$ which  implies:
1) $\nabla {\Xi}_{\bu_0}=0$ on each element  $T$;    2) ${\Xi}_{\bu_0}={\Xi}_{\bu_b}$ on each $\partial T$;  3) ${\Xi}_{\bu_0}=0$ on each $T$.  Since  $\nabla {\Xi}_{\bu_0}=0$ on each $T$,  it follows that
${\Xi}_{\bu_0}$ is constant on each element. Using the fact that ${\Xi}_{\bu_0}={\Xi}_{\bu_b}$ on each $\partial T$, it is continuous across element boundaries, and hence constant throughout $\Omega$. Given  ${\Xi}_{\bu_0}=0$ on each $T$, this constant must be zero, so ${\Xi}_{\bu_0}\equiv 0$ in   $\Omega$. Consequently, ${\Xi}_{\bu_b}\equiv 0$, implying ${\Xi}_{\bu_h}\equiv 0$.
  Substituting ${\Xi}_{\bu_h}\equiv 0$ into the first equation of \eqref{wgstokes} gives
$$
(\nabla_w\cdot \bv_h, {\Xi}_{p_h})=0, \qquad  \forall \bv_h\in V_h^0.
$$
By the inf-sup condition \eqref{infsup}, this implies
$\|{\Xi}_{p_h}\|=0$, i.e., $ {\Xi}_{p_h} \equiv  0$.

Thus, we conclude that
 $\bu_h^{(1)}\equiv \bu_h^{(2)}$ and $p_h^{(1)}\equiv p_h^{(2)}$. This proves the uniqueness of the solution and completes the proof of the Theorem.
\end{proof}

\section{Error Equations}
Let  $\bu$ and $p$ denote the exact solutions of the Brinkman equations \eqref{model}, and let  
  $\bu_h \in V_{h}^0$ and $p_h\in W_h$ be their numerical approximations obtained via the  WG  scheme \ref{PDWG1}.   We define the error functions  $e_{\bu_h}$ and $e_{p_h}$ as follows:
\begin{equation}\label{error} 
e_{\bu_h}=\bu-\bu_h,  \qquad e_{p_h}=p-p_h.
\end{equation}

\begin{lemma}\label{errorequa}
The error functions $e_{\bu_h}$ and $e_{p_h}$, as defined in \eqref{error}, satisfy the following error equations:
\begin{equation}\label{erroreqn}
\begin{split}
 (\nabla_{w} e_{\bu_h}, \nabla_w  \bv_h) -(\nabla_w \cdot \bv_h, e_{p_h})+(\kappa^{-1} e_{\bu_0}, \bv_0)  =&\ell_1 (\bu, \bv_h)+\ell_2 (  \bv_h, p),    \forall \bv_h\in V_h^0, \\(\nabla_w \cdot e_{\bu_h}, q_h)=&0, \qquad \forall q_h\in W_h,
\end{split}
\end{equation}
where 
$$
\ell_1 (\bu, \bv_h)=\sum_{T\in {\cal T}_h}  
  \langle \bv_b-\bv_0,   ({\cal Q}_h-I) \nabla \bu \cdot \bn \rangle_{\partial T},
$$ 
$$
\ell_2 (  \bv_h, p)=\sum_{T\in {\cal T}_h} -\langle ({\cal Q}_h -I)p, (\bv_b-\bv_0)\cdot \bn \rangle_{\partial T}.
$$
\end{lemma}
\begin{proof} Applying identity \eqref{pro1}, standard integration by parts, and choosing $\bvarphi= {\cal Q}_h \nabla \bu$ in  \eqref{2.4new}, we obtain 
\begin{equation}\label{term1}
\begin{split}
&\sum_{T\in {\cal T}_h} (\nabla_{w} \bu, \nabla_w  \bv_h)_T\\=&\sum_{T\in {\cal T}_h}\sum_{T\in {\cal T}_h} ({\cal Q}_h(\nabla \bu), \nabla_w  \bv_h)_T\\ 
=&\sum_{T\in {\cal T}_h}   (\nabla \bv_0,  {\cal Q}_h \nabla \bu)_T+
  \langle \bv_b-\bv_0,   {\cal Q}_h \nabla \bu \cdot \bn \rangle_{\partial T}\\
=& \sum_{T\in {\cal T}_h}(\nabla \bv_0,   \nabla \bu)_T+
  \langle \bv_b-\bv_0,   {\cal Q}_h \nabla \bu \cdot \bn \rangle_{\partial T}\\
  =& \sum_{T\in {\cal T}_h}-(  \bv_0,   \Delta \bu)_T+\langle \nabla \bu\cdot\bn, \bv_0\rangle_{\partial T}+
  \langle \bv_b-\bv_0,   {\cal Q}_h \nabla \bu \cdot \bn \rangle_{\partial T}\\
  =& \sum_{T\in {\cal T}_h}-(  \bv_0,   \Delta \bu)_T+ 
  \langle \bv_b-\bv_0,   ({\cal Q}_h-I) \nabla \bu \cdot \bn \rangle_{\partial T},
\end{split}
\end{equation}
where the boundary term $\sum_{T\in {\cal T}_h} \langle \nabla \bu\cdot\bn, \bv_b\rangle_{\partial T}=\langle \nabla \bu\cdot\bn, \bv_b\rangle_{\partial \Omega}=0$ since  $\bv_b=0$ on $\partial \Omega$.

Now, applying standard integration by parts to \eqref{divnew} with  $w={\cal Q}_h  p$, we get:
 \begin{equation}\label{termm}
     \begin{split}
  &\sum_{T\in {\cal T}_h} (\nabla_w \cdot \bv_h, p)_T\\=&
  \sum_{T\in {\cal T}_h} (\nabla_w \cdot \bv_h, {\cal Q}_h p)_T \\=& \sum_{T\in {\cal T}_h} (\nabla \cdot \bv_0,  {\cal Q}_h p)_T + \langle {\cal Q}_h p, (\bv_b-\bv_0)\cdot \bn \rangle_{\partial T}\\
  =& \sum_{T\in {\cal T}_h} (\nabla \cdot \bv_0,    p)_T + \langle {\cal Q}_h p, (\bv_b-\bv_0)\cdot \bn \rangle_{\partial T}\\
  =& \sum_{T\in {\cal T}_h}-( \bv_0,    \nabla p)_T 
+\langle p, \bv_0\cdot\bn\rangle_{\partial T}+ \langle {\cal Q}_h p, (\bv_b-\bv_0)\cdot \bn \rangle_{\partial T}\\    =& \sum_{T\in {\cal T}_h}-( \bv_0, \nabla p)_T + \langle ( {\cal Q}_h -I)p, (\bv_b-\bv_0)\cdot \bn \rangle_{\partial T},
\end{split}
 \end{equation}
where the boundary term   $\sum_{T\in {\cal T}_h} \langle p, \bv_b\cdot\bn\rangle_{\partial T}=\langle p, \bv_b\cdot\bn\rangle_{\partial \Omega}=0$ due to $\bv_b=0$ on $\partial \Omega$.

Subtracting \eqref{termm} from \eqref{term1}, and using the first equation in \eqref{model}, we find:
\begin{equation*}  
\begin{split}
&\sum_{T\in {\cal T}_h}(\nabla_{w} \bu, \nabla_w  \bv_h)_T -(\nabla_w \cdot \bv_h, p)_T+(\kappa ^{-1} \bu, \bv_0)_T \\=& \sum_{T\in {\cal T}_h}-(  \bv_0,   \Delta \bu)_T+ 
  \langle \bv_b-\bv_0,   ({\cal Q}_h-I) \nabla \bu \cdot \bn \rangle_{\partial T} + 
 ( \bv_0, \nabla p)_T \\&-\langle ({\cal Q}_h -I)p, (\bv_b-\bv_0)\cdot \bn \rangle_{\partial T}+(\kappa ^{-1} \bu, \bv_0)_T\\
 =&\sum_{T\in {\cal T}_h} (  \bv_0,   \bf)_T+ 
  \langle \bv_b-\bv_0,   ({\cal Q}_h-I) \nabla \bu \cdot \bn \rangle_{\partial T}   -\langle ({\cal Q}_h -I)p, (\bv_b-\bv_0)\cdot \bn \rangle_{\partial T}.
  \end{split}
\end{equation*}
Subtracting the first equation of  \eqref{WG}  from the above  yields  the first equation of \eqref{erroreqn}.

Note that using \eqref{pro2} and  the second equation of \eqref{model}, we have:
$$
0= (\nabla_w \cdot \bu, q_h)=\sum_{T\in {\cal T}_h}({\cal Q}_h \nabla\cdot \bu, q_h)_T=\sum_{T\in {\cal T}_h}(\nabla\cdot \bu, q_h)_T=0.
$$
Subtracting this from the second equation of \eqref{WG} yields the second equation in \eqref{erroreqn}, completing the proof.

\end{proof}

\section{Error Estimates} 
\begin{lemma}\cite{wg21, autostokes}\label{lem}
Let ${\cal T}_h$ be a finite element partition of the domain $\Omega$ satisfying the shape regularity condition stated in  \cite{wy3655}. Then, for any $0\leq s \leq 1$ and $1\leq m \leq k$, $1\leq  n \leq  2N+k-1$, the following estimates hold:
\begin{eqnarray}
\label{error1}
 \sum_{T\in {\cal T}_h} h_T^{2s}\|({\cal Q}_h-I)p\|^2_{s,T}&\leq& C  h^{2n}\|p\|^2_{n},\\
\label{error2}
\sum_{T\in {\cal T}_h}h_T^{2s}\|\bu- Q _0\bu\|^2_{s,T}&\leq& C h^{2(m+1)}\|\bu\|^2_{m+1},\\
\label{error3}\sum_{T\in {\cal T}_h}h_T^{2s}\|\nabla\bu-{\cal Q}_h(\nabla\bu)\|^2_{s,T}&\leq& C h^{2n}\|\bu\|^2_{n+1}.
\end{eqnarray}
 \end{lemma}
 \begin{lemma}\label{lemma1}
If   $\bu\in [H^{k+1}(\Omega)]^d$, then there exists a constant 
$C$
 such that 
\begin{equation}\label{erroresti1}
\3bar \bu-Q_h\bu \3bar \leq Ch^{k}\|\bu\|_{k+1}.
\end{equation}
\end{lemma}
\begin{proof}
From identity \eqref{2.4new}, the trace inequalities \eqref{tracein} and \eqref{trace}, and the Cauchy–Schwarz inequality, along with estimate \eqref{error2} for   $m=k$ and $s=0, 1$, we derive the following for any $\bvarphi\in [P_r(T)]^{d\times d}$,
\begin{equation*}
\begin{split}
&|\sum_{T\in {\cal T}_h} (\nabla_w (\bu-Q_h\bu), \bvarphi)_T|\\=&  | \sum_{T\in {\cal T}_h} (\nabla (\bu-Q_0\bu), \bvarphi)_T-\langle Q_b\bu-Q_0\bu, \bvarphi \cdot\bn\rangle_{\partial T}|\\
\leq & (\sum_{T\in {\cal T}_h} \|\nabla (\bu-Q_0\bu)\|_T )^{\frac{1}{2}}(\sum_{T\in {\cal T}_h}\|\bvarphi\|_T^2)^{\frac{1}{2}}+(\sum_{T\in {\cal T}_h} \| Q_b\bu-Q_0\bu\|_{\partial T} ^2)^{\frac{1}{2}} (\sum_{T\in {\cal T}_h}\|\bvarphi \cdot\bn\|_{\partial T}^2)^{\frac{1}{2}}\\\leq & (\sum_{T\in {\cal T}_h} \|\nabla (\bu-Q_0\bu)\|_T )^{\frac{1}{2}}(\sum_{T\in {\cal T}_h}\|\bvarphi\|_T^2)^{\frac{1}{2}}
\\&+(\sum_{T\in {\cal T}_h} h_T^{-1}\| \bu-Q_0\bu\|_{ T} ^2+h_T \| \bu-Q_0\bu\|_{1,T} ^2)^{\frac{1}{2}} (\sum_{T\in {\cal T}_h}h_T^{-1}\|\bvarphi \|_{T}^2)^{\frac{1}{2}}\\
\leq & Ch^k\|\bu\|_{k+1} (\sum_{T\in {\cal T}_h} \|\bvarphi \|_{T}^2)^{\frac{1}{2}}.
\end{split}
\end{equation*}
By taking $\bvarphi=\nabla_w (\bu-Q_h\bu)$, we obtain 
 \begin{equation}\label{a1}
\sum_{T\in {\cal T}_h} (\nabla_w (\bu-Q_h\bu), \nabla_w (\bu-Q_h\bu))_T\leq 
 Ch^{2k}\|\bu\|^2_{k+1}.  
 \end{equation}

Next,   applying the Cauchy–Schwarz inequality, we get:
$$
\sum_{T\in {\cal T}_h} (\kappa^{-1} (\bu-Q_0\bu), \bphi)_T\leq  (\sum_{T\in {\cal T}_h}  \|\kappa^{-1}(\bu-Q_0\bu)\|_T^2)^{\frac{1}{2}} (\sum_{T\in {\cal T}_h} \|\bphi\|_T^2)^{\frac{1}{2}}.$$
Letting $\bphi = \bu-Q_0\bu $ and using estimate \eqref{error2} with $m=k$ and $s=0$, we have
\begin{equation}\label{a2}
     \sum_{T\in {\cal T}_h}  (\kappa^{-1}(\bu-Q_0\bu), (\bu-Q_0\bu)) \leq Ch^{2k}\|\bu\|^2_{k+1}.
\end{equation} 

 Combining \eqref{a1} and \eqref{a2} completes the proof.
\end{proof}

\begin{lemma}\label{lemma2} \cite{autostokes}
If   $\bu\in [H^{k+1}(\Omega)]^d$, then there exists a constant 
$C$
 such that 
\begin{equation}\label{erroresti2}
(\sum_{T\in {\cal T}_h} \|\nabla_w\cdot (\bu-Q_h\bu)\|_T^2)^{
\frac{1}{2}
} \leq Ch^{k}\|\bu\|_{k+1}.
\end{equation}
\end{lemma}

\begin{lemma}
For any $\bu\in [H^{k+1}(\Omega)]^d$, $q\in H^{k}(\Omega)$, $\bv_h=\{\bv_0, \bv_b\}\in V_h^0$ and $q_h \in W_h$, the following estimates hold:
\begin{equation}\label{es1}
   |\ell_1(\bu, \bv_h)| \leq  Ch^k \|\bu\|_{k+1}\3bar \bv_h \3bar, 
\end{equation}
\begin{equation}\label{es2}
   |\ell_2(\bv_h, p)| \leq  Ch^k \|p\|_{k}\3bar \bv_h \3bar.
\end{equation}

\end{lemma}

\begin{proof}
Recall that ${\cal Q}_h$ denotes the $L^2$ projection operator onto the finite element space of piecewise polynomials of degree at most $2N+k-1\geq k$ on non-convex elements, and  $N+k-1\geq k$ on convex elements in the finite element partition.

To estimate \eqref{es1}, from the Cauchy–Schwarz inequality, the trace inequality \eqref{tracein}, the norm equivalence \eqref{normeq}, and the estimate \eqref{error3} with $n=k$, we obtain:
 \begin{equation*} 
\begin{split}
|\ell_1(\bu, \bv_h)|\leq &(\sum_{T\in {\cal T}_h}  
  h_T^{-1}\|\bv_b-\bv_0\|_{\partial T} ^2)^{\frac{1}{2}} (\sum_{T\in {\cal T}_h}  
 h_T \|({\cal Q}_h-I) \nabla \bu \cdot \bn\|_{\partial T} ^2)^{\frac{1}{2}} \\
\leq & \|\bv_h\|_{1, h} (\sum_{T\in {\cal T}_h}  
 \|({\cal Q}_h-I) \nabla \bu \cdot \bn\|_{T} ^2+ h_T^2 \|({\cal Q}_h-I) \nabla \bu \cdot \bn\|_{1, T} ^2)^{\frac{1}{2}}\\ 
 \leq & Ch^k \|\bu\|_{k+1}\3bar \bv_h \3bar,
\end{split}
\end{equation*}
which proves \eqref{es1}.

Similarly, to estimate \eqref{es2}, we apply the Cauchy–Schwarz inequality, the trace inequality \eqref{tracein}, the norm equivalence \eqref{normeq}, and the  estimate \eqref{error1} with $n=k$:
 \begin{equation*} 
\begin{split}
|\ell_2(\bv_h, p)|\leq &(\sum_{T\in {\cal T}_h}  
  h_T^{-1}\|\bv_b-\bv_0\|_{\partial T} ^2)^{\frac{1}{2}} (\sum_{T\in {\cal T}_h}  
 h_T \|({\cal Q}_h-I)p\|_{\partial T} ^2)^{\frac{1}{2}} \\
\leq & \|\bv_h\|_{1, h} (\sum_{T\in {\cal T}_h}  
 \|({\cal Q}_h-I)p\|_{T} ^2+ h_T^2 \|({\cal Q}_h-I)p\|_{1, T} ^2)^{\frac{1}{2}}\\ 
 \leq & Ch^k \|p\|_{k}\3bar \bv_h \3bar,
\end{split}
\end{equation*}which establishes \eqref{es2} and completes the proof.

\end{proof}
\begin{theorem}
Let the exact solution 
$(\bu, p)$
 of the Brinkman problem \eqref{model} satisfy 
$\bu\in [H^{k+1}(\Omega)]^d$ and $p\in H^k(\Omega)$. Let $(\bu_h, p_h)$ be the numerical solution of the Stablizer-Free WG scheme \ref{PDWG1}. Then, the following error estimate holds 
\begin{equation}\label{trinorm}
\3bar \bu-\bu_h\3bar +\|p-p_h\|\leq Ch^{k}(\|\bu\|_{k+1}+\|p\|_k).
\end{equation}
\end{theorem}
\begin{proof}
 Let $\bv_h=Q_h\bu-\bu_h$  in the first equation of \eqref{erroreqn}. Then:
  \begin{equation} \label{er}
\begin{split}
\3bar e_{\bu_h}\3bar ^2=& \sum_{T\in {\cal T}_h} (\nabla_w  e_{\bu_h} , \nabla_w   e_{\bu_h} )_T +(\kappa^{-1}  e_{\bu_0}, e_{\bu_0})_T\\
=& \sum_{T\in {\cal T}_h} (\nabla_w  e_{\bu_h} , \nabla_w    (\bu-Q_h\bu) )_T+ (\nabla_w  e_{\bu_h} , \nabla_w    ( Q_h\bu-\bu_h) )_T\\&+(\kappa^{-1}  e_{\bu_0}, \bu-Q_0\bu)_T+(\kappa^{-1}  e_{\bu_0}, Q_0\bu-\bu_0)_T\\
=& \sum_{T\in {\cal T}_h} (\nabla_w  e_{\bu_h} , \nabla_w    (\bu-Q_h\bu) )_T+ \ell_1(\bu, Q_h\bu-\bu_h) \\
&+\ell_2 (Q_h\bu-\bu_h, p)+\sum_{T\in {\cal T}_h} (\nabla_w \cdot(Q_h\bu-\bu_h), e_{p_h})_T\\&+(\kappa^{-1}  e_{\bu_0}, \bu-Q_0\bu)_T\\
 =&I_1+I_2+I_3+I_4+I_5. 
\end{split}
\end{equation}

We now proceed to estimate each term 
 $I_i$ for $i=1, \cdots, 5$.

Estimate of $I_1$:  Applying the Cauchy-Schwarz inequality and Lemma \ref{lemma1}, we have
  \begin{equation*} 
\begin{split}
\sum_{T\in {\cal T}_h} (\nabla_w  e_{\bu_h} , \nabla_w    (\bu-Q_h\bu) )_T\leq & \3bar   e_{\bu_h}\3bar \3bar       \bu-Q_h\bu  \3bar
\\ \leq & Ch^k\|\bu\|_{k+1}\3bar   e_{\bu_h}\3bar.
\end{split}
\end{equation*}

Estimate of $I_2$:   Choosing $\bv_h=Q_h\bu-\bu_h$ in \eqref{es1}, and applying Lemma \ref{lemma1} along with the triangle inequality, we obtain
\begin{equation*} 
\begin{split}
|\ell_1(\bu, Q_h\bu-\bu_h)|\leq & Ch^k\|\bu\|_{k+1}\3bar  Q_h\bu-\bu_h \3bar \\
\leq & Ch^k\|\bu\|_{k+1}(\3bar  Q_h\bu- \bu  \3bar+  \3bar   \bu-\bu_h \3bar)
\\  \leq & Ch^k\|\bu\|_{k+1}(h^k\|\bu\|_{k+1}+  \3bar   \bu-\bu_h \3bar).
\end{split}
\end{equation*}

Estimate of $I_3$:  Taking $\bv_h=Q_h\bu-\bu_h$ in \eqref{es2},and applying Lemma \ref{lemma1}, the triangle inequality, and Young's inequality, we obtain
\begin{equation*} 
\begin{split}
|\ell_2 (Q_h\bu-\bu_h, p)|\leq & Ch^k\|p\|_{k+1}\3bar  Q_h\bu-\bu_h \3bar \\
\leq & Ch^k\|p\|_{k+1}(\3bar  Q_h\bu- \bu  \3bar+  \3bar   \bu-\bu_h \3bar)
\\  \leq & Ch^k\|p\|_{k+1}(h^k\|\bu\|_{k+1}+  \3bar   \bu-\bu_h \3bar) 
 \\  \leq & C_1h^{2k}\|p\|^2_{k+1}+C_2 h^{2k}\|\bu\|^2_{k+1}+ Ch^k\|p\|_{k+1}  \3bar   \bu-\bu_h \3bar.
\end{split}
\end{equation*}

Estimate of $I_4$:    
From \eqref{pro3} and the second equation of \eqref{model}, we obtain
\begin{equation*} \label{ss2}
\begin{split}
 \sum_{T\in {\cal T}_h}(\nabla_w\cdot Q_h\bu,  p-p_h)_T 
=&\sum_{T\in {\cal T}_h}({\cal Q}_h(\nabla \cdot  \bu),   p-p_h)_T=0.
\end{split}
\end{equation*}
Using this identity  along with estimate \eqref{error1} (with $n=k$),  Lemma \ref{lemma2}, the identity $(\nabla_w \cdot \bu_h,  p-{\cal Q}_hp)_T=0$, the Cauchy-Schwarz inequality, \eqref{pro3},  the second equation in \eqref{model},  and Young's inequality, we get
 \begin{equation*} 
\begin{split}
&|\sum_{T\in {\cal T}_h}(\nabla_w \cdot(Q_h\bu-\bu_h), e_{p_h})_T|\\=&  |\sum_{T\in {\cal T}_h}(\nabla_w \cdot \bu_h,  p-p_h)_T  |\\
= &  |\sum_{T\in {\cal T}_h}(\nabla_w \cdot \bu_h,  p-{\cal Q}_hp)_T+ (\nabla_w \cdot \bu_h,  {\cal Q}_hp-p_h)_T |\\
= &  |\sum_{T\in {\cal T}_h}  (\nabla_w \cdot \bu_h,  {\cal Q}_hp-p_h)_T |\\
= &  |\sum_{T\in {\cal T}_h}  (\nabla_w \cdot (\bu_h-Q_h\bu),  {\cal Q}_hp-p_h)_T+(\nabla_w \cdot  Q_h\bu,  {\cal Q}_hp-p_h)_T |\\
= &  |\sum_{T\in {\cal T}_h}  (\nabla_w \cdot (\bu_h-Q_h\bu),  {\cal Q}_hp-p_h)_T+ ({\cal Q}_h(\nabla  \cdot   \bu),  {\cal Q}_hp-p_h)_T |
\\=&  |\sum_{T\in {\cal T}_h} (\nabla_w \cdot (\bu_h-Q_h\bu),  {\cal Q}_hp-p_h)_T  |\\
\leq & (\sum_{T\in {\cal T}_h} \|(\nabla_w \cdot(Q_h\bu-\bu)\| _T^2)^{\frac{1}{2}}  (\sum_{T\in {\cal T}_h} \|{\cal Q}_hp-p_h\| _T^2)^{\frac{1}{2}}  \\
\leq & C  h^k\|\bu\|_{k+1}h^k\|p\|_k  \\
\leq & C_1h^{2k}\|p\|^2_{k+1}+C_2 h^{2k}\|\bu\|^2_{k+1}.
\end{split}
\end{equation*}

Estimate of $I_5$:  Using the Cauchy-Schwarz inequality and estimate \eqref{erroresti1}, we obtain
\begin{equation*}
    \begin{split}
         \sum_{T\in {\cal T}_h} (\kappa^{-1}  e_{\bu_0}, \bu-Q_0\bu)_T& = \sum_{T\in {\cal T}_h} 
         (\kappa^{-1}  (\bu-\bu_0), \bu-Q_0\bu)_T \\
         & = \sum_{T\in {\cal T}_h} 
         (\kappa^{-1}  \bu, \bu-Q_0\bu)_T\\
            & = \sum_{T\in {\cal T}_h} 
         (\kappa^{-1}  (\bu-Q_0\bu), \bu-Q_0\bu)_T
         \\ &\leq \3bar \bu-Q_h\bu\3bar^2\leq   Ch^{2k} \|\bu\|^2_{k+1},
    \end{split}
\end{equation*}
  where we used the projection properties $(\bu_0, \bu-Q_0\bu)_T=0$ and $(Q_0\bu, \bu-Q_0\bu)_T=0$. 
  
  Substituting the bounds for $I_i$ for $i=1,\cdots, 5$ into \eqref{er}, we derive
 \begin{equation*} 
\begin{split}
 \3bar e_{\bu_h}\3bar^2 \leq  &Ch^k\|\bu\|_{k+1} \3bar e_{\bu_h}\3bar+Ch^k\|\bu\|_{k+1}(h^k\|\bu\|_{k+1}+  \3bar   \bu-\bu_h \3bar)\\&+  C_1h^{2k}\|p\|^2_{k+1}+C_2 h^{2k}\|\bu\|^2_{k+1}+ Ch^k\|p\|_{k+1}  \3bar   \bu-\bu_h \3bar. 
 \end{split}
\end{equation*}
   Thus, we obtain
\begin{equation}\label{euh}
   \3bar e_{\bu_h}\3bar\leq Ch^k(\|\bu\|_{k+1}+ \|p\|_{k}). 
\end{equation}

Finally, using the first equation of \eqref{erroreqn},  the identity $(\nabla_w \cdot\bv_h, {\cal Q}_h p-p )_T=0$, and the estimates \eqref{es1}, \eqref{es2}, along with the Cauchy-Schwarz inequality, we find
 \begin{equation*} 
\begin{split}
&|\sum_{T\in {\cal T}_h}(\nabla_w \cdot\bv_h, {\cal Q}_h p-p_h)_T|\\
\leq&|\sum_{T\in {\cal T}_h}(\nabla_w \cdot\bv_h, {\cal Q}_h p-p )_T+(\nabla_w \cdot\bv_h, p-p_h)_T|\\
\leq&
 |\ell_1(\bu, \bv_h)|+|\ell_2(  \bv_h, p)|+|(\nabla_w  e_{\bu_h}, \nabla_w\bv_h)|+|(\kappa^{-1}e_{\bu_0}, \bv_0)|\\
\leq &  Ch^k\|\bu\|_{k+1}\3bar  \bv_h\3bar+ Ch^k \|p\|_{k}\3bar \bv_h  \3bar +\3bar e_{\bu_h} \3bar \3bar \bv_h\3bar.
\end{split}
\end{equation*}
This, combining with the inf-sup condition \eqref{infsup} and the estimate \eqref{euh},  gives 
\begin{equation*} 
\begin{split}
\|{\cal Q}_h p-p_h \|\leq &C\frac{|\sum_{T\in {\cal T}_h}(\nabla_w \cdot\bv_h, {\cal Q}_h p-p_h)_T|}{\3bar \bv_h\3bar}\\
\leq & Ch^k\|\bu\|_{k+1}  + Ch^k \|p\|_{k}  +\3bar e_{\bu_h} \3bar 
\\
\leq & Ch^k(\|\bu\|_{k+1}  +   \|p\|_{k}). \end{split}
\end{equation*}
Combining this with \eqref{error1} (with $n=k$)  and the triangle inequality, we arrive at 
$$
\|e_{p_h}\|\leq \|{\cal Q}_h p-p_h\|+\|p-{\cal Q}_h p\|\leq Ch^k(\|\bu\|_{k+1}  +   \|p\|_{k}),
$$
which, together with \eqref{euh},  completes the proof.
\end{proof}

\section{    Error Estimates in $L^2$}
To obtain the error estimate in the $L^2$ norm, we utilize the standard duality  argument. Recall that the error in the velocity is denoted by $$e_{\bu_h}=\bu-\bu_h=\{e_{\bu_0}, e_{\bu_b}\}=\{\bu-\bu_0, \bu-\bu_b\}.$$
We introduce the quantity $\textbf{E}_h =Q_h\bu - \bu_h=\{\textbf{E}_0, \textbf{E}_b\}=\{Q_0\bu - \bu_0, Q_b\bu - \bu_b\}\in V_h^0$. To proceed, we consider the dual problem corresponding to the Brinkman system \eqref{model}. The goal is to find a pair $(\bw, q) \in [H^2(\Omega)]^d \times H^1(\Omega)$ satisfying:
\begin{equation}\label{dual}
\begin{split}
   - \Delta \bw+\kappa^{-1} \bw+\nabla q &=\textbf{E}_0, \qquad \text{in}\ \Omega,\\
\nabla \cdot\bw&=0,\qquad \text{in}\ \Omega,  \\
\bw&=0, \qquad \text{on}\ \partial\Omega.
    \end{split}
\end{equation}
We assume that the dual solution satisfies the regularity estimate:
\begin{equation}\label{regu2}
 \|\bw\|_2+\|q\|_1\leq C\|\textbf{E}_0\|.
 \end{equation}
 
 \begin{theorem}
Let $(\bu, p)\in [H^{k+1}(\Omega)]^d \times H^k(\Omega)$ be the exact solutions to the Brinkman problem \eqref{model}, and let
 $(\bu_h, p_h)\in V_h^0 \times  W_h$ be their numerical approximations obtained via the Stabilizer-Free Weak Galerkin Algorithm \ref{PDWG1}. Suppose the regularity condition \eqref{regu2} holds. Then, there exists a constant $C$ such that 
\begin{equation*}
\|\bu-\bu_0\|\leq Ch^{k+1}(\|\bu\|_{k+1}+\|p\|_k).
\end{equation*}
 \end{theorem}
 
 \begin{proof}
We test the first equation of the dual problem \eqref{dual} with  $\textbf{E}_0$ to obtain:
 \begin{equation}\label{e1}
 \begin{split}
 \|\textbf{E}_0\|^2 =&(- \Delta \bw+\kappa^{-1}\bw+\nabla q, \textbf{E}_0).
 \end{split}
 \end{equation}

 By applying identity \eqref{term1} with  $\bu=\bw$ and  $\bv_h=\textbf{E}_h$, we obtain:
$$
\sum_{T\in {\cal T}_h}(- \Delta \bw, \textbf{E}_0)_T=\sum_{T\in {\cal T}_h} (\nabla_w\bw, \nabla_w \textbf{E}_h)_T-\langle \textbf{E}_b-\textbf{E}_0, ({\cal Q}_h-I)\nabla \bw\cdot\bn \rangle_{\partial T}.
$$

Similarly, using \eqref{termm} with $p=q$ and $\bv_h=\textbf{E}_h$,  we have:
$$
\sum_{T\in {\cal T}_h} (\nabla q, \textbf{E}_0)=\sum_{T\in {\cal T}_h} -(\nabla_w \cdot \textbf{E}_h, {\cal Q}_hq)_T+\langle ({\cal Q}_h-I)q, (\textbf{E}_b-\textbf{E}_0)\cdot\bn\rangle_{\partial T}.
$$
Substituting these expressions into \eqref{e1} leads to:
\begin{equation}\label{e2}
 \begin{split}
 \|\textbf{E}_0\|^2 =&\sum_{T\in {\cal T}_h} (\nabla_w\bw, \nabla_w \textbf{E}_h)_T-\langle \textbf{E}_b-\textbf{E}_0, ({\cal Q}_h-I)\nabla \bw\cdot\bn \rangle_{\partial T}\\&-(\nabla_w \cdot \textbf{E}_h, {\cal Q}_hq)_T+\langle ({\cal Q}_h-I)q, (\textbf{E}_b-\textbf{E}_0)\cdot\bn\rangle_{\partial T}+(\kappa^{-1} \bw, \textbf{E}_0)_T.
 \end{split}
 \end{equation}

Using the second equation in \eqref{dual} along with identity \eqref{pro3}, we find:
  \begin{equation}\label{add}
 \begin{split}
  \sum_{T\in {\cal T}_h}(\nabla_w\cdot Q_h\bw, p-p_h)_T=  \sum_{T\in {\cal T}_h} ({\cal Q}_h(\nabla\cdot \bw), p-p_h)_T=0.
  \end{split}
 \end{equation}

Hence, applying \eqref{add} and the error equation \eqref{erroreqn} with $\bv_h=Q_h\bw$, we conclude:
  \begin{equation}\label{e3}
 \begin{split}
 &\|\textbf{E}_0\|^2 \\=&\sum_{T\in {\cal T}_h} (\nabla_w\bw, \nabla_w (\bu-\bu_h))_T -(\nabla_w\bw, \nabla_w (\bu-Q_h\bu))_T\\&-\langle \textbf{E}_b-\textbf{E}_0, ({\cal Q}_h-I)\nabla \bw\cdot\bn \rangle_{\partial T} -(\nabla_w \cdot (\bu-\bu_h), {\cal Q}_hq)_T\\&-(\nabla_w \cdot (Q_h\bu-\bu), {\cal Q}_hq)_T+\langle ({\cal Q}_h-I)q, (\textbf{E}_b-\textbf{E}_0)\cdot\bn\rangle_{\partial T}\\
 &+ (\kappa^{-1} \bw, Q_0\bu-\bu)_T+(\kappa^{-1} \bw,  \bu-\bu_0)_T\\
 =&\sum_{T\in {\cal T}_h} (\nabla_wQ_h\bw, \nabla_w (\bu-\bu_h))_T+(\nabla_w (\bw-Q_h\bw), \nabla_w (\bu-\bu_h))_T\\
 &- (\nabla_w  \bw, \nabla_w (\bu-Q_h\bu))_T-\langle \textbf{E}_b-\textbf{E}_0, ({\cal Q}_h-I)\nabla \bw\cdot\bn \rangle_{\partial T}\\&-(\nabla_w \cdot (\bu-\bu_h), {\cal Q}_hq)_T-(\nabla_w \cdot (Q_h\bu-\bu), {\cal Q}_hq)_T\\
 &+\langle ({\cal Q}_h-I)q, (\textbf{E}_b-\textbf{E}_0)\cdot\bn\rangle_{\partial T} -(\nabla_w\cdot Q_h\bw, p-p_h)_T \\&+ (\kappa^{-1} \bw, Q_0\bu-\bu)_T+(\kappa^{-1} Q_0 \bw,  \bu-\bu_0)_T\\
 =& \ell_1(\bu, Q_h\bw) +\ell_2(Q_h\bw, p) +\sum_{T\in {\cal T}_h}(\nabla_w (\bw-Q_h\bw), \nabla_w (\bu-\bu_h))_T\\
 &- (\nabla_w  \bw, \nabla_w (\bu-Q_h\bu))_T-\langle \textbf{E}_b-\textbf{E}_0, ({\cal Q}_h-I)\nabla \bw\cdot\bn \rangle_{\partial T}\\&-(\nabla_w \cdot (\bu-\bu_h),  q)_T-(\nabla_w \cdot (Q_h\bu-\bu),  q)_T\\
 &+\langle ({\cal Q}_h-I)q, (\textbf{E}_b-\textbf{E}_0)\cdot\bn\rangle_{\partial T}+ (\kappa^{-1}  \bw, Q_0\bu-\bu)_T\\
 =& \sum_{i=1}^9 I_i.
 \end{split}
 \end{equation}

Each term $I_i$ for $i=1, \cdots, 9$ is estimated as follows:
 
Estimate for $I_1$:  Applying the Cauchy-Schwarz inequality, the trace inequality \eqref{tracein}, estimate \eqref{error2} with  $m=1$,and estimate \eqref{error3} with  $n=k$, we obtain:
\begin{equation*} 
 \begin{split}
&|\ell_1(\bu, Q_h\bw)|\\=&|\sum_{T\in {\cal T}_h}  
  \langle Q_b\bw-Q_0\bw,   ({\cal Q}_h-I) \nabla \bu \cdot \bn \rangle_{\partial T}|\\
  \leq & (\sum_{T\in {\cal T}_h}  
  \| Q_b\bw-Q_0\bw\|^2_{\partial T})^
{\frac{1}{2}} (\sum_{T\in {\cal T}_h}\| ({\cal Q}_h-I) \nabla \bu \cdot \bn \|^2_{\partial T})
\\
\leq & (\sum_{T\in {\cal T}_h}  
  h_T^{-1}\|  \bw-Q_0\bw\|^2_{T}+h_T\|  \bw-Q_0\bw\|^2_{1, T})^
{\frac{1}{2}}\\
&\cdot(\sum_{T\in {\cal T}_h}  h_T^{-1}\| ({\cal Q}_h-I) \nabla \bu \cdot \bn \|^2_{T}+h_T \| ({\cal Q}_h-I) \nabla \bu \cdot \bn \|^2_{1, T})\\
\leq & Ch^{-1}h^2\|\bw\|_2 h^k\|\bu\|_{k+1}\leq Ch^{k+1}\|\bw\|_2  \|\bu\|_{k+1}.
  \end{split}
 \end{equation*}
Estimate for  $I_2$: Using the Cauchy-Schwarz inequality, the trace inequality \eqref{tracein}, estimate \eqref{error1} with  $n=k$,  and estimate \eqref{error2} with $m=1$, we get
\begin{equation*} 
 \begin{split}
&|\ell_2(Q_h\bw, p)| \\=&|\sum_{T\in {\cal T}_h} -\langle ({\cal Q}_h -I)p, (Q_b\bw-Q_0\bw)\cdot \bn \rangle_{\partial T}|\\
\leq & (\sum_{T\in {\cal T}_h}\| ({\cal Q}_h -I)p\|_{\partial T}^2)^\frac{1}{2}(\sum_{T\in {\cal T}_h} \|Q_b\bw-Q_0\bw)\cdot \bn \|_{\partial T}^2)^\frac{1}{2}\\
\leq & (\sum_{T\in {\cal T}_h}h_T^{-1}\| ({\cal Q}_h -I)p\|_{T}^2+h_T\| ({\cal Q}_h -I)p\|_{1, T}^2)^\frac{1}{2}\\&\cdot (\sum_{T\in {\cal T}_h} h_T^{-1}\|(\bw-Q_0\bw)\cdot \bn \|_{T}^2+h_T \|(\bw-Q_0\bw)\cdot \bn \|_{1, T}^2)^\frac{1}{2}\\
\leq &Ch^{-1}h^2\|\bw\|_2 h^k\|p\|_{k} \leq Ch^{k+1} \|\bw\|_2  \|p\|_{k}.
 \end{split}
 \end{equation*} 
Estimate for  $I_3$: Employing the Cauchy-Schwarz inequality, along with estimates \eqref{trinorm} and \eqref{erroresti1} (with $k=1$),  we derive
\begin{equation*} 
 \begin{split}
&|\sum_{T\in {\cal T}_h}(\nabla_w (\bw-Q_h\bw), \nabla_w (\bu-\bu_h))_T|\\
\leq & \3bar \bw-Q_h\bw\3bar \3bar \bu-\bu_h\3bar\\ \leq & Ch\|\bw\|_{2}  h^k(\|\bu\|_{k+1}+\|p\|_k)\leq Ch^{k+1}\|\bw\|_{2}   (\|\bu\|_{k+1}+\|p\|_k).
 \end{split}
 \end{equation*} 
Estimate for $I_4$:  Let 
 $Q^0$ denote  the $L^2$ projection onto $[P_0(T)]^{d\times d}$. For any $T\in {\cal T}_h$, it follows from \eqref{2.4} that
$$
(Q^0 (\nabla_w  \bw), \nabla_w (\bu-Q_h\bu))_T= -( \bu-Q_0\bu, \nabla\cdot(Q^0(\nabla_w  \bw)))_T+\langle \bu-Q_b\bu, Q^0(\nabla_w  \bw) \cdot\bn\rangle_{\partial T}=0.
$$
Using this identity along with the Cauchy-Schwarz inequality,   \eqref{pro1}, and \eqref{erroresti1}, we obtain:
\begin{equation*}
\begin{split}
&|\sum_{T\in {\cal T}_h}(\nabla_w  \bw, \nabla_w (\bu-Q_h\bu))_T| \\
=&|\sum_{T\in {\cal T}_h}(\nabla_w  \bw-Q^0 (\nabla_w  \bw), \nabla_w (\bu-Q_h\bu))_T| \\
=&|\sum_{T\in {\cal T}_h}({\cal Q}_h(\nabla   \bw)-Q^0 ({\cal Q}_h(\nabla   \bw)), \nabla_w (\bu-Q_h\bu))_T|  \\
\leq & (\sum_{T\in {\cal T}_h}\|{\cal Q}_h(\nabla   \bw)-Q^0 ({\cal Q}_h(\nabla   \bw))\|_T^2)^\frac{1}{2}(\sum_{T\in {\cal T}_h}\| \nabla_w (\bu-Q_h\bu)\|_T^2)^\frac{1}{2}\\
\leq & Ch\|{\cal Q}_h(\nabla   \bw)\|_1 h^k\|\bu\|_{k+1}
\\ \leq & Ch^{k+1}\|\bw\|_2  \|\bu\|_{k+1}.
 \end{split}
 \end{equation*} 
Estimate for $I_5$:  Applying the Cauchy-Schwarz inequality, trace inequality \eqref{tracein}, norm equivalence \eqref{normeq}, estimate \eqref{error3} with $n=1$, the triangle inequality, and the error estimates \eqref{erroresti1} and \eqref{trinorm}, we have:
\begin{equation*}
\begin{split}
&|\sum_{T\in {\cal T}_h}\langle \textbf{E}_b-\textbf{E}_0, ({\cal Q}_h-I)\nabla \bw\cdot\bn \rangle_{\partial T}|\\
\leq & (\sum_{T\in {\cal T}_h}h_T^{-1}\|\textbf{E}_b-\textbf{E}_0\|_{\partial T}^2)^\frac{1}{2}  (\sum_{T\in {\cal T}_h} h_T\|({\cal Q}_h-I)\nabla \bw\cdot\bn \|_{\partial T}^2)^\frac{1}{2}\\
\leq & \|\textbf{E}_h\|_{1, h} (\sum_{T\in {\cal T}_h} \|({\cal Q}_h-I)\nabla \bw\cdot\bn \|_{T}^2+h_T^2\|({\cal Q}_h-I)\nabla \bw\cdot\bn \|_{1,T}^2)^\frac{1}{2}\\
 \leq & \3bar \textbf{E}_h\3bar  h\|\bw\|_2
 \\
  \leq & (\3bar Q_h\bu-\bu\3bar+\3bar \bu-\bu_h\3bar ) h\|\bw\|_2\\
 \leq & Ch^{k+1}(\|\bu\|_{k+1}+\|p\|_k)\|\bw\|_2.
 \end{split}
 \end{equation*} 
Estimate for  $I_6$: From the second equation in \eqref{model} and the properties \eqref{pro2}–\eqref{pro3}, we have:
\begin{equation*} 
    \nabla_w \cdot \bu={\cal Q}_h \nabla  \cdot \bu=0, \qquad \nabla_w \cdot Q_h\bu={\cal Q}_h \nabla  \cdot \bu=0.
\end{equation*}
This gives \begin{equation}\label{qq}\nabla_w \cdot  \bu=\nabla_w \cdot Q_h\bu =0.\end{equation}
Let  ${\cal Q}_h^{k-1}$ denote the $L^2$ projection onto $P_{k-1}(T)$.  
Using the second equation in \eqref{erroreqn} by letting $q_h={\cal Q}_h^{k-1}q\in W_h$, together with the Cauchy-Schwarz inequality, the estimate \eqref{error1} with $n=1$,   \eqref{qq}, and the error bound \eqref{erroresti2}, we obtain:
\begin{equation*}
\begin{split}
& |\sum_{T\in {\cal T}_h}|(\nabla_w \cdot (\bu-\bu_h),  q)_T|\\
& |\sum_{T\in {\cal T}_h}|(\nabla_w \cdot (\bu-\bu_h),  q-{\cal Q}_h^{k-1}q )_T|\\=&|\sum_{T\in {\cal T}_h}| ( \nabla_w \cdot (Q_h\bu-\bu_h),  q-{\cal Q}_h^{k-1}q )_T|\\
\leq & (\sum_{T\in {\cal T}_h}\|\nabla_w \cdot (Q_h\bu-\bu_h)\|_T^2)^{\frac{1}{2}}(\sum_{T\in {\cal T}_h}\|q-{\cal Q}_h^{k-1}q\|_T^2)^{\frac{1}{2}}\\
\leq & Ch^{k}\|\bu\|_{k+1}h \|q\|_1.
\end{split}
 \end{equation*} 
Estimate for  $I_7$:  Let  $Q^0$ be the  $L^2$ projection onto $P_0(T)$. From \eqref{div}, for any $T\in {\cal T}_h$, we have:
$$
(Q^0 q, \nabla_w \cdot ( Q_h\bu-\bu ))_T= -( Q_0\bu- \bu, \nabla (Q^0 q))_T+\langle  Q_b\bu-\bu, Q^0 q\cdot\bn\rangle_{\partial T}=0.
$$
Using this identity along with the Cauchy-Schwarz inequality and the estimate \eqref{erroresti2}, it follows that:
\begin{equation*}
\begin{split}
& |\sum_{T\in {\cal T}_h}(\nabla_w \cdot (Q_h\bu-\bu),  q)_T|\\
\leq &  |\sum_{T\in {\cal T}_h}(\nabla_w \cdot (Q_h\bu-\bu),  q-Q^0q)_T|
\\
\leq & (\sum_{T\in {\cal T}_h}\|\nabla_w \cdot (Q_h\bu-\bu)\|_T^2)^\frac{1}{2}  (\sum_{T\in {\cal T}_h}\| q-Q^0q\|_T^2)^\frac{1}{2}\\
\leq & Ch^k\|\bu\|_{k+1 } h\|q\|_1.
 \end{split}
 \end{equation*} 
Estimate for  $I_8$:  By the Cauchy-Schwarz inequality, the trace inequality \eqref{tracein}, norm equivalence \eqref{normeq}, estimate \eqref{error1} with $n=1$, the triangle inequality, and the error bounds \eqref{erroresti1} and \eqref{trinorm}, we obtain:
\begin{equation*}
\begin{split}
&|\sum_{T\in {\cal T}_h}\langle ({\cal Q}_h-I)q, (\textbf{E}_b-\textbf{E}_0)\cdot\bn\rangle_{\partial T}|\\
 \leq&  (\sum_{T\in {\cal T}_h}h_T\|({\cal Q}_h-I)q\|_{\partial T}  ^2)^\frac{1}{2}  (\sum_{T\in {\cal T}_h}h_T^{-1}\|(\textbf{E}_b-\textbf{E}_0)\cdot\bn\|_{\partial T}^2)^\frac{1}{2} \\
  \leq&  (\sum_{T\in {\cal T}_h} \|({\cal Q}_h-I)q\|_{T}  ^2+h_T^2\|({\cal Q}_h-I)q\|_{1, T}  ^2)^\frac{1}{2}  \|\textbf{E}_h\|_{1, h}\\
  \leq &  Ch\|q\|_1 (\3bar Q_h\bu-\bu\3bar+\3bar  \bu-\bu_h\3bar)\\
   \leq &  Ch\|q\|_1 h^k(\|\bu\|_{k+1}+\|p\|_k). \\
   \end{split}
 \end{equation*} 
Estimate for $I_9$:  Applying the Cauchy-Schwarz inequality and estimate \eqref{error2}  with $m=k$ yields:
\begin{equation*}
    \begin{split}
      | \sum_{T\in {\cal T}_h} (\kappa^{-1} \bw, Q_0\bu-\bu)_T| 
      \leq & \Big(\sum_{T\in {\cal T}_h}\|\kappa^{-1}  \bw \|_T^2\Big)^{\frac{1}{2}}\Big(\sum_{T\in {\cal T}_h}\|Q_0\bu-\bu\|_T^2\Big)^{\frac{1}{2}}\\
      \leq & C \|\bw\|_0 h^{k+1}\|\bu\|_{k+1}\\
      \leq & Ch^{k+1}\|\bw\|_2  \|\bu\|_{k+1}.
    \end{split}
\end{equation*}
Substituting the estimates for $I_i$ ($i=1, \cdots, 9$) into \eqref{e3} and applying the regularity assumption \eqref{regu2}, we conclude:
\begin{equation*}
\|\textbf{E}_0\|\leq Ch^{k+1}(\|\bu\|_{k+1}+\|p\|_k).
\end{equation*}
Using the triangle inequality then yields the final result:
\begin{equation*}
\|\bu-\bu_0\|\leq \|\bu-Q_0\bu\|+\|\textbf{E}_0\|\leq Ch^{k+1}(\|\bu\|_{k+1}+\|p\|_k).
\end{equation*}

This completes the proof of the theorem. 
\end{proof}

\section{Numerical experiments}
In the 2D test,  we solve the Brinkman problem \eqref{weak} on the unit square domain $\Omega=
  (0,1)\times (0,1)$, where $\kappa=1$.
The exact solution is chosen as
\an{\label{s-2d} \b u=\p{ -8 (x^2-2 x^3 + x^4 ) ( y -3 y^2 + 2y^3 )\\
   \ \, 8 ( y -3 x^2 + 2 x^3 ) (x^2-2 x^3 + x^4 ) }, \quad  p=(x-\frac 12)^3.  
}

In the first computation, we compute the weak Galerkin finite element solutions by
  the algorithm \eqref{WG}, on triangular meshes shown in Figure \ref{f-2-1}.
We use the stabilizer-free method where we take $r=k+1$  in \eqref{wg-k} in computing the weak
  gradient.
Naturally, we take $r=k-1$  in \eqref{wd-k} in computing the weak
  divergence.
The results are listed in Table \ref{t1} where we have the
  optimal order of convergence for all variables and in all norms.
  
\begin{figure}[H]
 \begin{center}\setlength\unitlength{1.0pt}
\begin{picture}(380,120)(0,0)
  \put(15,108){$G_1$:} \put(125,108){$G_2$:} \put(235,108){$G_3$:} 
  \put(0,-420){\includegraphics[width=380pt]{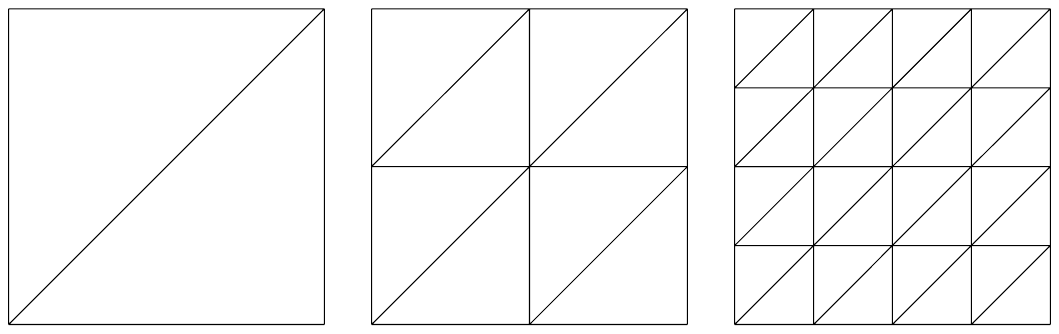}}  
 \end{picture}\end{center}
\caption{The triangular meshes for the computation in Table \ref{t1}. }\label{f-2-1}
\end{figure}

  \begin{table}[H]
  \caption{ Error profile for computing \eqref{s-2d} on meshes shown in Figure \ref{f-2-1}.} \label{t1}
\begin{center}  
   \begin{tabular}{c|rr|rr|rr}  
 \hline 
$G_i$ &  $ \|Q_h\b u - \b u_h \| $ & $O(h^r)$ &  $ \3bar Q_h \b u-\bw\3bar $ & $O(h^r)$ 
   &  $ \| p -p_h \| $ & $O(h^r)$   \\ \hline 
&\multicolumn{6}{c}{By the $P_1$-$P_1$/$P_0$ weak Galerkin finite element \eqref{V-h} and \eqref{W-h} }\\
 \hline 
 5&     0.317E-3 &  1.9&     0.185E-1 &  1.0&     0.850E-2 &  0.9 \\
 6&     0.808E-4 &  2.0&     0.926E-2 &  1.0&     0.431E-2 &  1.0 \\
 7&     0.203E-4 &  2.0&     0.463E-2 &  1.0&     0.216E-2 &  1.0 \\
 \hline 
&\multicolumn{6}{c}{By the $P_2$-$P_2$/$P_1$ weak Galerkin finite element \eqref{V-h} and \eqref{W-h} }\\
 \hline  
 5&     0.474E-5 &  3.1&     0.106E-2 &  2.0&     0.521E-3 &  2.0 \\
 6&     0.573E-6 &  3.0&     0.267E-3 &  2.0&     0.129E-3 &  2.0 \\
 7&     0.708E-7 &  3.0&     0.668E-4 &  2.0&     0.320E-4 &  2.0 \\
 \hline 
&\multicolumn{6}{c}{By the $P_3$-$P_3$/$P_2$ weak Galerkin finite element \eqref{V-h} and \eqref{W-h} }\\
 \hline  
 4&     0.319E-5 &  4.0&     0.389E-3 &  2.8&     0.160E-3 &  2.8 \\
 5&     0.194E-6 &  4.0&     0.505E-4 &  2.9&     0.199E-4 &  3.0 \\
 6&     0.120E-7 &  4.0&     0.641E-5 &  3.0&     0.239E-5 &  3.1 \\
 \hline 
&\multicolumn{6}{c}{By the $P_4$-$P_4$/$P_3$ weak Galerkin finite element \eqref{V-h} and \eqref{W-h} }\\
 \hline  
 3&     0.685E-5 &  4.5&     0.465E-3 &  3.5&     0.201E-3 &  3.6 \\
 4&     0.231E-6 &  4.9&     0.314E-4 &  3.9&     0.123E-4 &  4.0 \\
 5&     0.743E-8 &  5.0&     0.204E-5 &  3.9&     0.722E-6 &  4.1 \\
\hline 
\end{tabular} \end{center}  \end{table}

We compute again the weak Galerkin finite element solutions by
  the algorithm \eqref{WG}, but on non-convex polygon meshes shown in Figure \ref{f-2-2}.
We use the stabilizer-free method where we take $r=k+3$  in \eqref{wg-k} in computing the weak
  gradient. Again we take $r=k-1$  in \eqref{wd-k} in computing the weak
  divergence.
The results are listed in Table \ref{t2} where we have the
  optimal order of convergence for all variables and in all norms.
  
\begin{figure}[H]
 \begin{center}\setlength\unitlength{1.0pt}
\begin{picture}(380,120)(0,0)
  \put(15,108){$G_1$:} \put(125,108){$G_2$:} \put(235,108){$G_3$:} 
  \put(0,-420){\includegraphics[width=380pt]{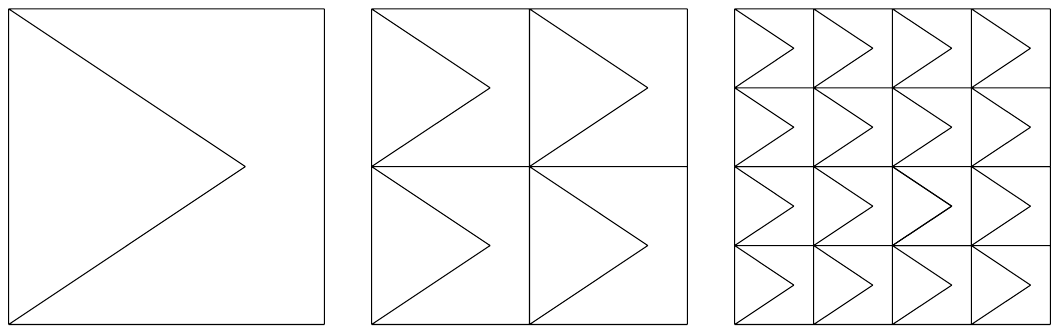}}  
 \end{picture}\end{center}
\caption{The non-convex polygon meshes for the computation in Table \ref{t2}. }\label{f-2-2}
\end{figure}

  \begin{table}[H]
  \caption{ Error profile for computing \eqref{s-2d} on meshes shown in Figure \ref{f-2-2}.} \label{t2}
\begin{center}  
   \begin{tabular}{c|rr|rr|rr}  
 \hline 
$G_i$ &  $ \|Q_h\b u - \b u_h \| $ & $O(h^r)$ &  $ \3bar Q_h \b u-\bw\3bar $ & $O(h^r)$ 
   &  $ \| p -p_h \| $ & $O(h^r)$   \\ \hline 
&\multicolumn{6}{c}{By the $P_1$-$P_1$/$P_0$ weak Galerkin finite element \eqref{V-h} and \eqref{W-h} }\\
 \hline 
 5&     0.186E-2 &  1.7&     0.519E-1 &  1.1&     0.897E-2 &  1.3 \\
 6&     0.501E-3 &  1.9&     0.250E-1 &  1.1&     0.290E-2 &  1.6 \\
 7&     0.128E-3 &  2.0&     0.124E-1 &  1.0&     0.984E-3 &  1.6 \\
 \hline 
&\multicolumn{6}{c}{By the $P_2$-$P_2$/$P_1$ weak Galerkin finite element \eqref{V-h} and \eqref{W-h} }\\
 \hline  
 4&     0.124E-3 &  3.3&     0.155E-1 &  2.8&     0.484E-2 &  2.1 \\
 5&     0.117E-4 &  3.4&     0.364E-2 &  2.1&     0.106E-2 &  2.2 \\
 6&     0.121E-5 &  3.3&     0.911E-3 &  2.0&     0.251E-3 &  2.1 \\
 \hline 
&\multicolumn{6}{c}{By the $P_3$-$P_3$/$P_2$ weak Galerkin finite element \eqref{V-h} and \eqref{W-h} }\\
 \hline  
 2&     0.923E-2 &  6.0&     0.821E+0 &  4.5&     0.202E-1 &  1.2 \\
 3&     0.179E-3 &  5.7&     0.269E-1 &  4.9&     0.471E-2 &  2.1 \\
 4&     0.683E-5 &  4.7&     0.140E-2 &  4.3&     0.605E-3 &  3.0 \\
 \hline 
&\multicolumn{6}{c}{By the $P_4$-$P_4$/$P_3$ weak Galerkin finite element \eqref{V-h} and \eqref{W-h} }\\
 \hline  
 1&     0.290E+0 &  0.0&     0.115E+2 &  0.0&     0.523E-1 &  0.0 \\
 2&     0.229E-2 &  7.0&     0.226E+0 &  5.7&     0.126E-1 &  2.0 \\
 3&     0.238E-4 &  6.6&     0.409E-2 &  5.8&     0.989E-3 &  3.7 \\
\hline 
\end{tabular} \end{center}  \end{table}

In Table \ref{t-3}, we compute the weak Galerkin finite element solutions 
    on non-convex polygon meshes shown in Figure \ref{f-2-5}.
We use the stabilizer-free method where we take $r=k+2$  in \eqref{wg-k} in computing the weak
  gradient. 
We get the  optimal order of convergence for all variables and in all norms.

\begin{figure}[H]
 \begin{center}\setlength\unitlength{1.0pt}
\begin{picture}(380,120)(0,0)
  \put(15,108){$G_1$:} \put(125,108){$G_2$:} \put(235,108){$G_3$:} 
  \put(0,-420){\includegraphics[width=380pt]{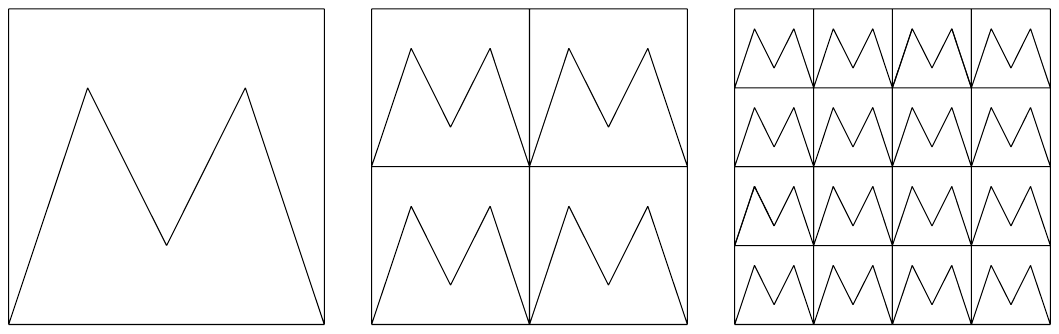}}  
 \end{picture}\end{center}
\caption{The non-convex polygon meshes for the computation in Table \ref{t-3}. }\label{f-2-5}
\end{figure}

  \begin{table}[H]
  \caption{ Error profile for computing \eqref{s-2d} on meshes shown in Figure \ref{f-2-5}.} \label{t-3}
\begin{center}  
   \begin{tabular}{c|rr|rr|rr}  
 \hline 
$G_i$ &  $ \|Q_h\b u - \b u_h \| $ & $O(h^r)$ &  $ \3bar Q_h \b u-\bw\3bar $ & $O(h^r)$ 
   &  $ \| p -p_h \| $ & $O(h^r)$   \\ \hline 
&\multicolumn{6}{c}{By the $P_1$-$P_1$/$P_0$ weak Galerkin finite element \eqref{V-h} and \eqref{W-h} }\\
 \hline 
 5&     0.536E-3 &  1.9&     0.501E-1 &  1.0&     0.299E-2 &  1.3 \\
 6&     0.137E-3 &  2.0&     0.251E-1 &  1.0&     0.124E-2 &  1.3 \\
 7&     0.343E-4 &  2.0&     0.125E-1 &  1.0&     0.578E-3 &  1.1 \\
 \hline 
&\multicolumn{6}{c}{By the $P_2$-$P_2$/$P_1$ weak Galerkin finite element \eqref{V-h} and \eqref{W-h} }\\
 \hline  
 4&     0.867E-4 &  3.0&     0.126E-1 &  2.0&     0.173E-2 &  2.3 \\
 5&     0.102E-4 &  3.1&     0.318E-2 &  2.0&     0.348E-3 &  2.3 \\
 6&     0.123E-5 &  3.0&     0.800E-3 &  2.0&     0.780E-4 &  2.2 \\ 
 \hline 
&\multicolumn{6}{c}{By the $P_3$-$P_3$/$P_2$ weak Galerkin finite element \eqref{V-h} and \eqref{W-h} }\\
 \hline  
 3&     0.754E-4 &  4.2&     0.807E-2 &  3.8&     0.164E-2 &  2.5 \\
 4&     0.476E-5 &  4.0&     0.964E-3 &  3.1&     0.210E-3 &  3.0 \\
 5&     0.294E-6 &  4.0&     0.123E-3 &  3.0&     0.252E-4 &  3.1 \\
 \hline 
&\multicolumn{6}{c}{By the $P_4$-$P_4$/$P_3$ weak Galerkin finite element \eqref{V-h} and \eqref{W-h} }\\
 \hline  
 1&     0.312E-1 &  0.0&     0.121E+1 &  0.0&     0.121E-1 &  0.0 \\
 2&     0.305E-3 &  6.7&     0.279E-1 &  5.4&     0.324E-2 &  1.9 \\
 3&     0.729E-5 &  5.4&     0.925E-3 &  4.9&     0.232E-3 &  3.8 \\
\hline 
\end{tabular} \end{center}  \end{table}

In the last 2D computation,  we compute the weak Galerkin finite element solutions 
    on non-convex polygon meshes shown in Figure \ref{f-2-6}.
We use the stabilizer-free method where we take $r=k+3$  in \eqref{wg-k} in computing the weak
  gradient. 
We get the  optimal order of convergence for all variables and in all norms in Table \ref{t4}.

\begin{figure}[H]
 \begin{center}\setlength\unitlength{1.0pt}
\begin{picture}(380,120)(0,0)
  \put(15,108){$G_1$:} \put(125,108){$G_2$:} \put(235,108){$G_3$:} 
  \put(0,-420){\includegraphics[width=380pt]{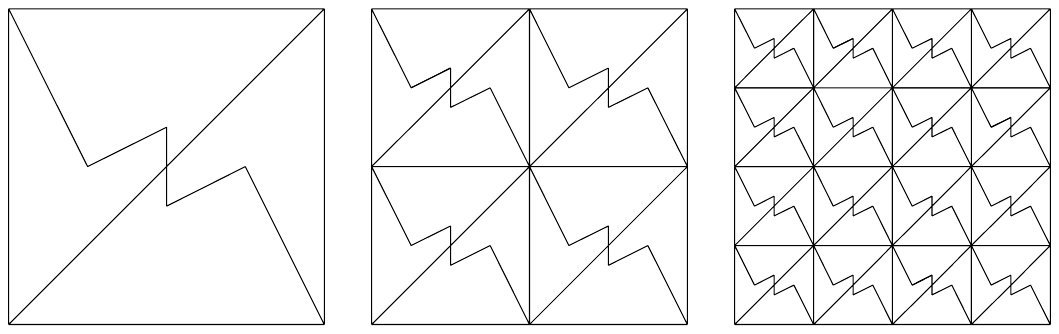}}  
 \end{picture}\end{center}
\caption{The non-convex polygon meshes for the computation in Table \ref{t4}. }\label{f-2-6}
\end{figure}

  \begin{table}[H]
  \caption{ Error profile for computing \eqref{s-2d} on meshes shown in Figure \ref{f-2-6}.} \label{t4}
\begin{center}  
   \begin{tabular}{c|rr|rr|rr}  
 \hline 
$G_i$ &  $ \|Q_h\b u - \b u_h \| $ & $O(h^r)$ &  $ \3bar Q_h \b u-\bw\3bar $ & $O(h^r)$ 
   &  $ \| p -p_h \| $ & $O(h^r)$   \\ \hline 
&\multicolumn{6}{c}{By the $P_1$-$P_1$/$P_0$ weak Galerkin finite element \eqref{V-h} and \eqref{W-h} }\\
 \hline 
 5&     0.312E-3 &  2.0&     0.344E-1 &  1.0&     0.837E-2 &  1.0 \\
 6&     0.787E-4 &  2.0&     0.172E-1 &  1.0&     0.420E-2 &  1.0 \\
 7&     0.197E-4 &  2.0&     0.863E-2 &  1.0&     0.210E-2 &  1.0 \\
 \hline 
&\multicolumn{6}{c}{By the $P_2$-$P_2$/$P_1$ weak Galerkin finite element \eqref{V-h} and \eqref{W-h} }\\
 \hline  
 4&     0.272E-4 &  3.0&     0.502E-2 &  1.9&     0.146E-2 &  1.9 \\
 5&     0.332E-5 &  3.0&     0.127E-2 &  2.0&     0.365E-3 &  2.0 \\
 6&     0.412E-6 &  3.0&     0.317E-3 &  2.0&     0.908E-4 &  2.0 \\ 
 \hline 
&\multicolumn{6}{c}{By the $P_3$-$P_3$/$P_2$ weak Galerkin finite element \eqref{V-h} and \eqref{W-h} }\\
 \hline  
 2&     0.205E-3 &  5.4&     0.141E-1 &  4.7&     0.585E-2 &  1.9 \\
 3&     0.136E-4 &  3.9&     0.193E-2 &  2.9&     0.779E-3 &  2.9 \\
 4&     0.844E-6 &  4.0&     0.247E-3 &  3.0&     0.924E-4 &  3.1 \\
 \hline 
&\multicolumn{6}{c}{By the $P_4$-$P_4$/$P_3$ weak Galerkin finite element \eqref{V-h} and \eqref{W-h} }\\
 \hline  
 1&     0.140E-2 &  0.0&     0.651E-1 &  0.0&     0.713E-2 &  0.0 \\
 2&     0.337E-4 &  5.4&     0.285E-2 &  4.5&     0.102E-2 &  2.8 \\
 3&     0.103E-5 &  5.0&     0.175E-3 &  4.0&     0.599E-4 &  4.1 \\
\hline 
\end{tabular} \end{center}  \end{table}

In the 3D test,  we solve the Brinkman problem \eqref{weak} on the unit cube domain $\Omega=
  (0,1)\times (0,1)\times (0,1)$, where $\kappa=1$.
The exact solution is chosen as
\an{\label{s-3d}\ad{
   \b u &=\p{ -2^{10} x^2 (1-x)^2 y^2 (1-y)^2 ( z -3 z^2 + 2z^3 )\qquad \\ \\
           \ \,2^{10} x^2 (1-x)^2 y^2 (1-y)^2 ( z -3 z^2 + 2z^3 ) \qquad \\ \\
   \ \,2^{10} \Big(( x -3 x^2 + 2x^3 )(y-y^2)^2 \qquad\qquad\qquad\\
       - (x-x^2)^2( y -3 x^2 + 2 x^3 )\Big) (z-z^2)^2 }, \\ 
    p&=10(3y^2-2y^3-y).   }
}

We first compute the weak Galerkin finite element solutions for the 3D problem \eqref{s-3d} by
  the algorithm \eqref{WG}, on tetrahedral meshes shown in Figure \ref{f-3-1}.
We use the stabilizer-free method where we take $r=k+1$  in \eqref{wg-k} in computing the weak
  gradient. 
The results are listed in Table \ref{t5} where we have the
  optimal order of convergence for all variables and in all norms.
  
\begin{figure}[H]
 \begin{center}\setlength\unitlength{1.0pt}
\begin{picture}(380,120)(0,0)
  \put(35,108){$G_1$:} \put(145,108){$G_2$:} \put(255,108){$G_3$:} 
  \put(0,-420){\includegraphics[width=380pt]{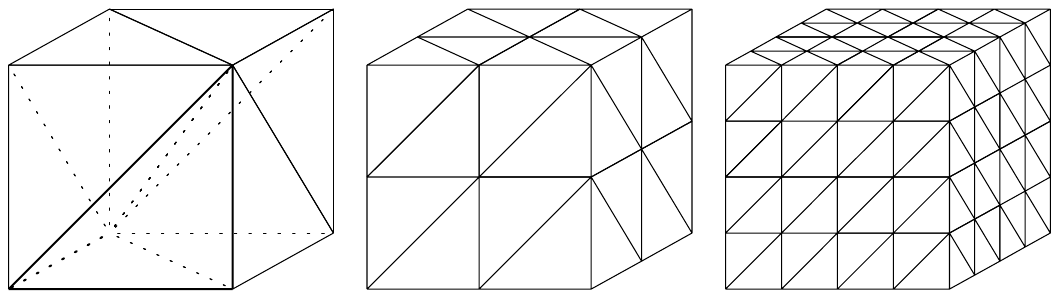}}  
 \end{picture}\end{center}
\caption{The triangular meshes for the computation in Table \ref{t5}. }\label{f-3-1}
\end{figure}

  \begin{table}[H]
  \caption{ Error profile for computing \eqref{s-3d} on meshes shown in Figure \ref{f-3-1}.} \label{t5}
\begin{center}  
   \begin{tabular}{c|rr|rr|rr}  
 \hline 
$G_i$ &  $ \|Q_h\b u - \b u_h \| $ & $O(h^r)$ &  $ \3bar Q_h \b u-\bw\3bar $ & $O(h^r)$ 
   &  $ \| p -p_h \| $ & $O(h^r)$   \\ \hline 
&\multicolumn{6}{c}{By the $P_1$-$P_1$/$P_0$ weak Galerkin finite element \eqref{V-h} and \eqref{W-h} }\\
 \hline 
 2 &    0.871E+1 &0.00 &    0.891E+0 &0.31 &    0.474E+0 &0.00 \\
 3 &    0.578E+1 &0.59 &    0.397E+0 &1.17 &    0.179E+0 &1.41 \\
 4 &    0.317E+1 &0.86 &    0.145E+0 &1.45 &    0.562E-1 &1.67 \\
 \hline 
&\multicolumn{6}{c}{By the $P_2$-$P_2$/$P_1$ weak Galerkin finite element \eqref{V-h} and \eqref{W-h} }\\
 \hline  
 1 &    0.690E+1 &0.00 &    0.772E+0 &0.00 &    0.790E+0 &0.00 \\
 2 &    0.343E+1 &1.01 &    0.406E+0 &0.93 &    0.242E+0 &1.70 \\
 3 &    0.133E+1 &1.36 &    0.931E-1 &2.12 &    0.434E-1 &2.48 \\ 
 \hline 
&\multicolumn{6}{c}{By the $P_3$-$P_3$/$P_2$ weak Galerkin finite element \eqref{V-h} and \eqref{W-h} }\\
 \hline  
 1 &    0.527E+1 &0.00 &    0.864E+0 &0.00 &    0.701E+0 &0.00 \\
 2 &    0.166E+1 &1.67 &    0.161E+0 &2.43 &    0.100E+0 &2.81 \\
 3 &    0.302E+0 &2.46 &    0.316E-1 &2.35 &    0.118E-1 &3.08 \\
\hline 
\end{tabular} \end{center}  \end{table}

We next compute the weak Galerkin finite element solutions for \eqref{s-3d} 
   on non-convex polyhedral meshes shown in Figure \ref{f-3-3}.
We use the stabilizer-free method where we take $r=k+2$  in \eqref{wg-k} in computing the weak
  gradient. 
The results are listed in Table \ref{t6} where we have the
  optimal order of convergence for all variables and in all norms.
  
\begin{figure}[H]
 \begin{center}\setlength\unitlength{1.0pt}
\begin{picture}(380,120)(0,0)
  \put(35,108){$G_1$:} \put(145,108){$G_2$:} \put(255,108){$G_3$:} 
  \put(0,-420){\includegraphics[width=380pt]{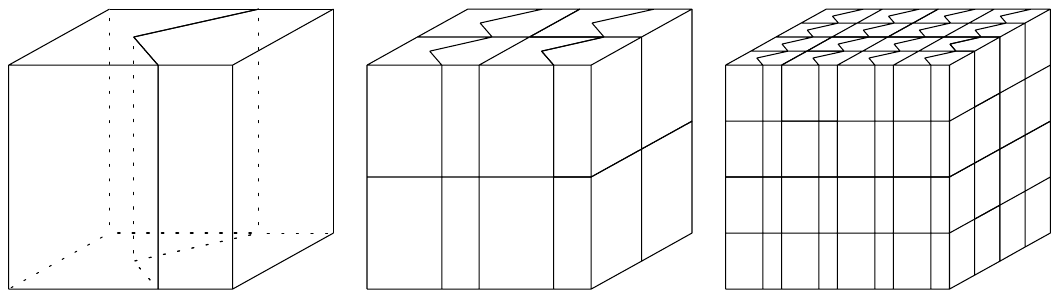}}  
 \end{picture}\end{center}
\caption{The triangular meshes for the computation in Table \ref{t6}. }\label{f-3-3}
\end{figure}

  \begin{table}[H]
  \caption{ Error profile for computing \eqref{s-3d} on meshes shown in Figure \ref{f-3-3}.} \label{t6}
\begin{center}  
   \begin{tabular}{c|rr|rr|rr}  
 \hline 
$G_i$ &  $ \|Q_h\b u - \b u_h \| $ & $O(h^r)$ &  $ \3bar Q_h \b u-\bw\3bar $ & $O(h^r)$ 
   &  $ \| p -p_h \| $ & $O(h^r)$   \\ \hline 
&\multicolumn{6}{c}{By the $P_1$-$P_1$/$P_0$ weak Galerkin finite element \eqref{V-h} and \eqref{W-h} }\\
 \hline 
 2 &    0.116E+0 &1.51 &    0.366E+1 &0.00 &    0.566E+0 &0.00 \\
 3 &    0.253E-1 &2.20 &    0.142E+1 &1.36 &    0.220E+0 &1.36 \\
 4 &    0.533E-2 &2.25 &    0.433E+0 &1.71 &    0.678E-1 &1.70 \\
 \hline 
&\multicolumn{6}{c}{By the $P_2$-$P_2$/$P_1$ weak Galerkin finite element \eqref{V-h} and \eqref{W-h} }\\
 \hline  
 2 &    0.144E+0 &2.37 &    0.528E+1 &1.59 &    0.157E+1 &2.30 \\
 3 &    0.161E-1 &3.16 &    0.111E+1 &2.25 &    0.269E+0 &2.55 \\
 4 &    0.352E-2 &2.20 &    0.147E+0 &2.92 &    0.320E-1 &3.07 \\
 \hline 
&\multicolumn{6}{c}{By the $P_3$-$P_3$/$P_2$ weak Galerkin finite element \eqref{V-h} and \eqref{W-h} }\\
 \hline  
 1 &    0.120E+1 &0.00 &    0.414E+2 &0.00 &    0.168E+2 &0.00 \\
 2 &    0.206E+0 &2.54 &    0.994E+1 &2.06 &    0.341E+1 &2.30 \\
 3 &    0.772E-2 &4.74 &    0.639E+0 &3.96 &    0.198E+0 &4.10 \\
\hline 
\end{tabular} \end{center}  \end{table}
 
  We compute the weak Galerkin finite element solutions for \eqref{s-3d} 
   on non-convex polyhedral meshes shown in Figure \ref{f-3-4}, in Table \ref{t7}.
We use the stabilizer-free method where we take $r=k+3$  in \eqref{wg-k} in computing the weak
  gradient.  
  
\begin{figure}[H]
 \begin{center}\setlength\unitlength{1.0pt}
\begin{picture}(380,120)(0,0)
  \put(35,108){$G_1$:} \put(145,108){$G_2$:} \put(255,108){$G_3$:} 
  \put(0,-420){\includegraphics[width=380pt]{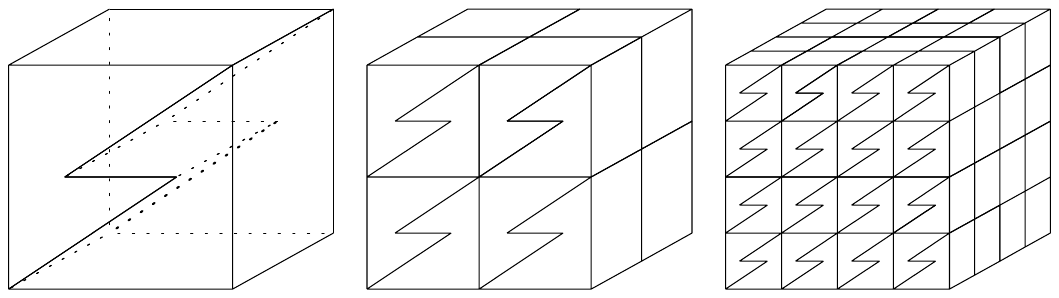}}  
 \end{picture}\end{center}
\caption{The triangular meshes for the computation in Table \ref{t7}. }\label{f-3-4}
\end{figure}

  \begin{table}[H]
  \caption{ Error profile for computing \eqref{s-3d} on meshes shown in Figure \ref{f-3-4}.} \label{t7}
\begin{center}  
   \begin{tabular}{c|rr|rr|rr}  
 \hline 
$G_i$ &  $ \|Q_h\b u - \b u_h \| $ & $O(h^r)$ &  $ \3bar Q_h \b u-\bw\3bar $ & $O(h^r)$ 
   &  $ \| p -p_h \| $ & $O(h^r)$   \\ \hline 
&\multicolumn{6}{c}{By the $P_1$-$P_1$/$P_0$ weak Galerkin finite element \eqref{V-h} and \eqref{W-h} }\\
 \hline 
 2 &    0.116E+0 &1.51 &    0.366E+1 &0.00 &    0.566E+0 &0.00 \\
 3 &    0.253E-1 &2.20 &    0.142E+1 &1.36 &    0.220E+0 &1.36 \\
 4 &    0.533E-2 &2.25 &    0.433E+0 &1.71 &    0.678E-1 &1.70 \\
 \hline 
&\multicolumn{6}{c}{By the $P_2$-$P_2$/$P_1$ weak Galerkin finite element \eqref{V-h} and \eqref{W-h} }\\
 \hline  
 2 &    0.144E+0 &2.37 &    0.528E+1 &1.59 &    0.157E+1 &2.30 \\
 3 &    0.161E-1 &3.16 &    0.111E+1 &2.25 &    0.269E+0 &2.55 \\
 4 &    0.352E-2 &2.20 &    0.147E+0 &2.92 &    0.320E-1 &3.07 \\
 \hline 
&\multicolumn{6}{c}{By the $P_3$-$P_3$/$P_2$ weak Galerkin finite element \eqref{V-h} and \eqref{W-h} }\\
 \hline  
 1 &    0.120E+1 &0.00 &    0.414E+2 &0.00 &    0.168E+2 &0.00 \\
 2 &    0.206E+0 &2.54 &    0.994E+1 &2.06 &    0.341E+1 &2.30 \\
 3 &    0.772E-2 &4.74 &    0.639E+0 &3.96 &    0.198E+0 &4.10 \\
\hline 
\end{tabular} \end{center}  \end{table}

  We compute the weak Galerkin finite element solutions for \eqref{s-3d} 
   on polyhedral meshes shown in Figure \ref{f-3-5}, in Table \ref{t8}.
We use the stabilizer-free method where we take $r=k+1$  in \eqref{wg-k} in computing the weak
  gradient.  
  
\begin{figure}[H]
 \begin{center}\setlength\unitlength{1.0pt}
\begin{picture}(380,120)(0,0)
  \put(35,108){$G_1$:} \put(145,108){$G_2$:} \put(255,108){$G_3$:} 
  \put(0,-420){\includegraphics[width=380pt]{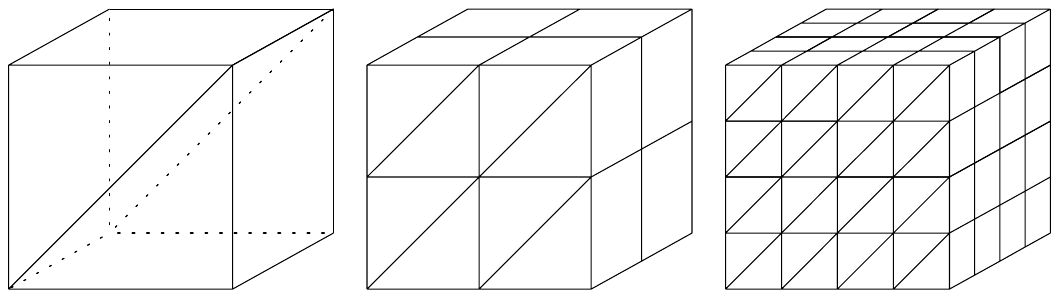}}  
 \end{picture}\end{center}
\caption{The triangular meshes for the computation in Table \ref{t8}. }\label{f-3-5}
\end{figure}

  \begin{table}[H]
  \caption{ Error profile for computing \eqref{s-3d} on meshes shown in Figure \ref{f-3-5}.} \label{t8}
\begin{center}  
   \begin{tabular}{c|rr|rr|rr}  
 \hline 
$G_i$ &  $ \|Q_h\b u - \b u_h \| $ & $O(h^r)$ &  $ \3bar Q_h \b u-\bw\3bar $ & $O(h^r)$ 
   &  $ \| p -p_h \| $ & $O(h^r)$   \\ \hline 
&\multicolumn{6}{c}{By the $P_1$-$P_1$/$P_0$ weak Galerkin finite element \eqref{V-h} and \eqref{W-h} }\\
 \hline 
 2 &    0.879E-1 &1.97 &    0.333E+1 &0.00 &    0.512E+0 &0.22 \\
 3 &    0.317E-1 &1.47 &    0.130E+1 &1.36 &    0.315E+0 &0.70 \\
 4 &    0.748E-2 &2.08 &    0.434E+0 &1.58 &    0.116E+0 &1.44 \\ 
 \hline 
&\multicolumn{6}{c}{By the $P_2$-$P_2$/$P_1$ weak Galerkin finite element \eqref{V-h} and \eqref{W-h} }\\
 \hline  
 2 &    0.107E+0 &1.90 &    0.341E+1 &0.93 &    0.127E+1 &1.78 \\
 3 &    0.103E-1 &3.39 &    0.756E+0 &2.17 &    0.215E+0 &2.56 \\
 4 &    0.340E-2 &1.59 &    0.111E+0 &2.76 &    0.301E-1 &2.83 \\
 \hline 
&\multicolumn{6}{c}{By the $P_3$-$P_3$/$P_2$ weak Galerkin finite element \eqref{V-h} and \eqref{W-h} }\\
 \hline  
 1 &    0.326E+0 &0.00 &    0.732E+1 &0.00 &    0.857E+1 &0.00 \\
 2 &    0.103E+0 &1.66 &    0.475E+1 &0.62 &    0.158E+1 &2.44 \\
 3 &    0.519E-2 &4.31 &    0.359E+0 &3.73 &    0.120E+0 &3.72 \\
\hline 
\end{tabular} \end{center}  \end{table}

As the last test, compute the weak Galerkin finite element solutions for \eqref{s-3d} 
   on non-convex polyhedral meshes shown in Figure \ref{f-3-7}, in Table \ref{t9}.
We use the stabilizer-free method where we take $r=k+2$  in \eqref{wg-k} in computing the weak
  gradient.  
  
\begin{figure}[H]
 \begin{center}\setlength\unitlength{1.0pt}
\begin{picture}(380,120)(0,0)
  \put(35,108){$G_1$:} \put(145,108){$G_2$:} \put(255,108){$G_3$:} 
  \put(0,-420){\includegraphics[width=380pt]{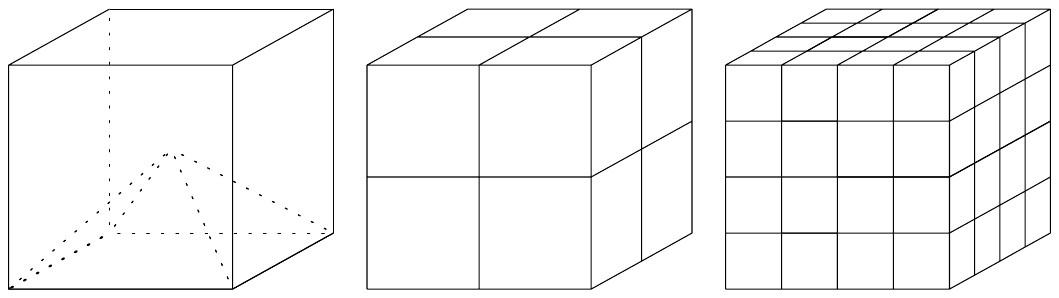}}  
 \end{picture}\end{center}
\caption{The triangular meshes for the computation in Table \ref{t9}. }\label{f-3-7}
\end{figure}

  \begin{table}[H]
  \caption{ Error profile for computing \eqref{s-3d} on meshes shown in Figure \ref{f-3-7}.} \label{t9}
\begin{center}  
   \begin{tabular}{c|rr|rr|rr}  
 \hline 
$G_i$ &  $ \|Q_h\b u - \b u_h \| $ & $O(h^r)$ &  $ \3bar Q_h \b u-\bw\3bar $ & $O(h^r)$ 
   &  $ \| p -p_h \| $ & $O(h^r)$   \\ \hline 
&\multicolumn{6}{c}{By the $P_1$-$P_1$/$P_0$ weak Galerkin finite element \eqref{V-h} and \eqref{W-h} }\\
 \hline 
 2 &    0.164E+0 &0.77 &    0.858E+1 &0.00 &    0.320E+1 &0.00 \\
 3 &    0.969E-1 &0.76 &    0.350E+1 &1.29 &    0.109E+1 &1.56 \\
 4 &    0.362E-1 &1.42 &    0.122E+1 &1.52 &    0.248E+0 &2.13 \\
 \hline 
&\multicolumn{6}{c}{By the $P_2$-$P_2$/$P_1$ weak Galerkin finite element \eqref{V-h} and \eqref{W-h} }\\
 \hline  
 2 &    0.198E+0 &2.81 &    0.843E+1 &1.82 &    0.396E+1 &2.28 \\
 3 &    0.313E-1 &2.66 &    0.244E+1 &1.79 &    0.786E+0 &2.33 \\
 4 &    0.325E-2 &3.27 &    0.367E+0 &2.74 &    0.110E+0 &2.84 \\
 \hline 
&\multicolumn{6}{c}{By the $P_3$-$P_3$/$P_2$ weak Galerkin finite element \eqref{V-h} and \eqref{W-h} }\\
 \hline  
 1 &    0.216E+1 &0.00 &    0.455E+2 &0.00 &    0.485E+2 &0.00 \\
 2 &    0.283E+0 &2.93 &    0.154E+2 &1.56 &    0.685E+1 &2.82 \\
 3 &    0.153E-1 &4.21 &    0.145E+1 &3.41 &    0.568E+0 &3.59 \\
\hline 
\end{tabular} \end{center}  \end{table}

\end{document}